\documentclass[preprint,12pt]{elsarticle}

\usepackage{amssymb}
\usepackage{amsfonts}
\usepackage{amsmath}
\usepackage{amsbsy}
\usepackage{theorem}
\usepackage{color}

\newtheorem{theorem}{Theorem}
\newtheorem{proposition}{Proposition}
\newtheorem{lemma}{Lemma}

\newtheorem{remark}{Remark}

\journal{Statistics}

\begin{document}

\begin{frontmatter}



\title{{\bfseries Functional Analysis of Variance for Hilbert-Valued
Multivariate Fixed  Effect Models}\\
The final version of this draft will be published in Statistics.
}


\author{M.D. Ruiz-Medina}

\address{Department of Statistics and O.R.\\ Faculty of Sciences\\
Campus Fuente Nueva s/n\\University of Granada\\18071, Granada\\
e-mail: mruiz@ugr.es (corresponding author)}

\begin{abstract}
This paper presents new results on Functional Analysis of Variance
for  fixed effect models with correlated Hilbert-valued  Gaussian error components. The geometry of the Reproducing Kernel Hilbert Space
(RKHS) of the error term is considered in the computation of the total sum of squares,
the residual sum of squares, and the sum of squares due to the
regression.  Under suitable linear transformation of the correlated
functional data, the distributional characteristics of these
statistics, their moment generating and characteristic functions,  are derived. Fixed effect linear hypothesis testing is
finally formulated in the Hilbert-valued multivariate Gaussian
context considered.
\end{abstract}


\end{frontmatter}

\section{Introduction}

An extensive literature  on functional data analysis
techniques emerges
 in the last few decades. In
particular, in the functional linear regression context, one can
refer to the papers
   \cite{CaiHall06}; \cite{CardotFerraty03}; \cite{CardotSarda11};  \cite{Chiou04}; \cite{Crambes09}; \cite{Cuevas02}; \cite{Kokoszkaetal08}, among others.
  We particularly refer to the flexible approach recently presented in \cite{Ferraty13} to approximate the regression
function in the case of a functional predictor and
a scalar response, based on the Projection Pursuit Regression
principle. Specifically, an additive decomposition, which exploits
the most interesting projections of the prediction variable to
explain the response, is derived. This approach can be used as an
exploratory tool for the analysis of functional dataset, and  the
dimensionality problem is overcome.
   In the functional nonparametric regression
   framework we refer to    \cite{Ferraty06} and \cite{Ferraty11}, among others.
 Asymptotic
results, in particular, uniform consistency,  in the  purely
nonparametric context are derived in \cite{Kudraszowa13}  for a kNN
generalized regression estimator.
   For more
   details we refer to the  reader to the  nice summary on the statistics theory with  functional data in \cite{Cuevas14}, and  the references therein. New branches of the functional statistical
    theory in a univariate and multivariate framework  are collected in  \cite{Bongiorno14}.

 Functional
Analysis of Variance (FANOVA) extends  the classical ANOVA methods,
allowing the analysis of high-dimensional data  with a  functional
background. Due to the vast existing  literature on functional data
statistical analysis techniques, FANOVA models have recently gained
popularity and related literature has been steadily growing. For
comprehensive reviews we refer, for example, to   \cite{Ramsay} and
\cite{Stone}. In particular, FANOVA model fitting
 and  its component estimation have been addressed in several papers  (see
  \cite{Angelini};
\cite{Gu}; \cite{Huang}; \cite{Kaufman}; \cite{Kaziska}; \cite{Lin};
\cite{Wahba}, among others).

In the context of  hypothesis testing from functional data,
\cite{Faraway} discussed the difficulties of generalizing the ideas
of multivariate testing procedures to the functional data analysis
context.   A powerful overall test for functional hypothesis
testing, based on the decomposition of the original functional data
into Fourier and wavelet series expansions, is proposed in
\cite{Fan2}. In this paper,  the adaptive Neyman and wavelet
thresholding procedures of \cite{Fan1} are respectively  applied to
the resulting empirical Fourier and wavelet coefficients. The
general philosophy of the presented methodology exploits the
sparsity of the signal representation in the Fourier and wavelet
domains,  allowing a significant dimension reduction. Somewhat
similar approaches were considered in \cite{Dette} and
\cite{Eubank}. In \cite{Guo}
 a maximum likelihood ratio based test is suggested for functional
variance components in mixed-effect FANOVA models. The procedures
presented in \cite{Fan2} are applied to the mixed-effect FANOVA
model in \cite{Spitzner}. An alternative asymptotic
approach, inspired in classical ANOVA tests, is derived in
\cite{Cuevas04}, for studying the equality of the functional means
from k independent samples of functional data. In \cite{Abramovich}
and \cite{Abramovich2} the testing problem in the mixed-effect
functional analysis of variance models is addressed, developing
asymptotically optimal (minimax) testing procedures for the
significance of functional global trend, and the functional fixed
effects, from  the empirical wavelet coefficients of the data (see
also \cite{Antoniadis07}). Statistical shape analysis methods are
applied in \cite{Munk} to developing a neighborhood hypothesis
testing procedure to establish that a mean is in a specified $\delta
$-neighborhood.
 Recently, in the context of functional data defined by curves, considering the
$L^{2}$-norm,  an up-to-date overview of hypothesis testing methods
for functional data analysis is provided in \cite{Zhang}, including
 functional ANOVA,
functional linear models with functional responses, heteroscedastic
ANOVA for functional data, and hypothesis tests for the equality of
covariance functions, among other related topics.

Most of  the above-cited papers are based on
dimension reduction techniques, applying numerical projection and
smoothing methods. Classical ANOVA results are then considered for
projections in an univariate and multivariate framework. For
example, smoothing splines ANOVA (SS-ANOVA) has the restriction of
considering the time as an additional factor which implies
independence in time. This restriction was removed in
\cite{BrumbackRice98} and \cite{Gu}, using the RKHS theory.  This
theory is also applied in the present paper to remove independence
assumption on the functional zero-mean Gaussian error  vector  term. Since the
aim of this paper is to provide explicit results within the
infinite-dimensional probability distribution setting, we restrict
our attention to the Gaussian case. It is well-known that Gaussian
measures on Hilbert spaces can be identified with infinite products
of independence real-valued Gaussian measures (see, for example,
\cite{Da Prato}). Hence, one can work with infinite series of
independent real-valued random variables. Indeed, such an
identification, jointly with the acute formulation  of Cram\'er-Wold
theorem derived in \cite{Cuesta-Albertos07}, can be applied to
implement the multiway ANOVA methodology for functional data
proposed in \cite{Cuesta-Albertos10}, which does not require the
Gaussian assumption although the hypothesis of independence is maintained (see \cite{Da Prato}). Here, we also apply the
results in \cite{Da Prato}, on characteristic functions of quadratic
functionals constructed from Gaussian measures on Hilbert spaces,
for the derivation of the probability distribution  of the
functional components of variance, and of the  test statistics
formulated for linear hypothesis testing.

In all above-cited papers,
a Hilbert-valued formulation of the traditional analysis of
variance has tended to be missing, since the Functional Analysis of
Variance derived in \cite{Zoglat}. Specifically, in \cite{Zoglat},
the following $L_{2}([0,1])$-valued fixed effect model is
considered:
 \begin{equation}\mathbf{Y}(t) = \mathbf{X}\boldsymbol{\beta }(t) + \sigma \boldsymbol{\varepsilon} (t),\quad t\in [0,1],\label{fem}\end{equation} \noindent
where $\mathbf{X}=(x_{ij})$ is a $n\times p$ fixed effect design matrix. The response  $\mathbf{Y}$ is a $n$-dimensional
vector of independent  Gaussian $L_{2}([0,1])$-valued components with $E[\mathbf{Y}]=\mathbf{X}\boldsymbol{\beta },$ and with $L_{2}([0,1])$ denoting  the
space of square integrable functions on $[0,1].$
The unknown functional parameter $\boldsymbol{\beta }(t)$ takes its
values in the space $L_{2}^{p}([0,1])$ of vectorial functions with square integrable components on the interval $[0,1].$
The error term
 $\boldsymbol{\varepsilon}$ is an   $n$-dimensional
$L_{2}([0,1])$-valued zero-mean Gaussian random variable with
 covariance matrix operator
\begin{eqnarray}& &E\left[[\varepsilon_{1}(\cdot),\dots,\varepsilon_{n}(\cdot)]^{T}
[\varepsilon_{1}(\cdot),\dots,\varepsilon_{n}(\cdot)]\right]=\mbox{diag}(R)
\nonumber\\
& &=\left[\begin{array}{ccc}E[\varepsilon_{1}\otimes \varepsilon_{1}]&,\dots,&
E[\varepsilon_{1}\otimes \varepsilon_{n}]\\
\vdots & \vdots & \vdots \\
E[\varepsilon_{n}\otimes \varepsilon_{1}]&,\dots,&
E[\varepsilon_{n}\otimes \varepsilon_{n}]
\end{array}\right],\label{covest1}
\end{eqnarray}
\noindent where  $[\cdot ]^{T}$ denotes
transposition, $\mbox{diag}(R)$ is a diagonal matrix operator with
non-null functional  entries, in the diagonal, given by the compact
and self-adjoint operator $R,$ defined on $L_{2}([0,1]),$ i.e.,
$E[\varepsilon_{i}\otimes \varepsilon_{j}]=\delta_{i,j}R,$ with
$\delta $ denoting the Kronecker delta function.
 Here, $\sigma $ represents a scale
parameter. Note that, in the Gaussian case,
$E\|\boldsymbol{\varepsilon}\|_{H^{n}}^{2}=n(\mbox{trace}(R))<\infty
,$ with $\mbox{trace}(\cdot )$ denoting the trace of an operator,
which implies that $R$ is in the trace class (see, for example,
\cite{Da Prato}, Chapter 1).

This paper extends the results derived in  \cite{Zoglat} to an arbitrary
Hilbert space $H$ (not necessarily given by $L_{2}([0,1])$), and to the case where
$\boldsymbol{\varepsilon}$ has correlated $H$-valued zero-mean Gaussian components.
Specifically,  a generalized
least-squares estimator of $H^{p}-$valued parameter
$\boldsymbol{\beta }$ is obtained. The functional mean-square error
is computed  in the RKHS norm. It is proved that, for an orthogonal fixed effect design
matrix, the
statistics minimizing the functional quadratic loss function  takes its values
in the functional parameter space $H^{p}.$
The analysis developed here is referred to a common  orthonormal eigenvector system, which is assumed to be known, providing the spectral diagonalization  of the covariance
operators of the error components. This assumption is satisfied, for example,  by the system of stochastic differential or
pseudodifferential equations
 introduced in Section \ref{Exm}, with fixed effect $H$-valued parameters. In this case, the common  orthonormal basis of eigenvectors of $H$ can be determined from
 the differential or pseudodifferential operators defining  such a system of equations (see
 Section \ref{Exm} below).

Another important issue addressed in this paper is the
construction  of a  matrix operator providing a suitable functional linear
 transformation of our observed $H^{n}$-valued response (see Section
\ref{secW}), in order to
ensure the almost surely finiteness of the total sum of squares, the
sum of squares due to regression, and the residual sum of squares.
Under this transformation, the moment generating and characteristic
functionals of these three statistics are derived. Linear
hypothesis testing is also addressed, in terms of a suitable matrix operator  class defining a linear transformation  of the $H$-valued components of $\boldsymbol{\beta },$ to test some contrasts.

The outline of the paper is as follows. Section \ref{Sec2}
introduces the analyzed Hilbert-valued multivariate Gaussian fixed
effect model with  correlated  error components.  In Section
\ref{sec3}, the generalized least-squares estimator of the
$H^{p}$-valued fixed effect parameter $\boldsymbol{\beta }$ is
derived, providing  sufficient conditions  for the almost surely
finiteness of its $H^{p}$-norm. The transformed functional data
model is constructed in Section \ref{sec4}.  The almost surely finiteness of the functional
components of variance is then proved.  Their
moment generating and characteristic functionals are derived  in Section \ref{sec5}. Linear hypothesis testing is
addressed in Section \ref{FLHT} in a multivariate Hilbert-valued
Gaussian framework.   Final comments are provided in Section
\ref{Sec6}.
\section{The model}
\label{Sec2} Let $H$ be a real separable Hilbert space endowed with
the inner product $\left\langle \cdot,\cdot\right\rangle_{H}.$
Consider the following Hilbert-valued multivariate fixed effect
model:
\begin{equation}\mathbf{Y}(\cdot)=\mathbf{X}\boldsymbol{\beta
}(\cdot)+\sigma \boldsymbol{\varepsilon}(\cdot),\label{eq1} \end{equation}
\noindent where
$\mathbf{Y}(\cdot)=[Y_{1}(\cdot),\dots,Y_{n}(\cdot)]^{T}$ is an $H^{n}$-valued Gaussian  random variable,
with $E[\mathbf{Y}]=\mathbf{X}\boldsymbol{\beta }$. The $H^{n}$-valued error term $\boldsymbol{\varepsilon}(\cdot)=[\varepsilon_{1}(\cdot),\dots,\varepsilon_{n}(\cdot)]^{T}$ is such that
$E[\boldsymbol{\varepsilon}]=\underline{\mathbf{0}},$ and has covariance matrix  operator \begin{eqnarray}
&&\hspace*{-1cm}\mathbf{R}_{\boldsymbol{\varepsilon}\boldsymbol{\varepsilon}}=E\left[[\varepsilon_{1}(\cdot),\dots,\varepsilon_{n}(\cdot)]^{T}
[\varepsilon_{1}(\cdot),\dots,\varepsilon_{n}(\cdot)]\right]
\nonumber\\
& &\hspace*{-1cm}=\left[\begin{array}{ccc}E[\varepsilon_{1}\otimes \varepsilon_{1}]&,\dots,&
E[\varepsilon_{1}\otimes \varepsilon_{n}]\\
\vdots & \vdots & \vdots \\
E[\varepsilon_{n}\otimes \varepsilon_{1}]&,\dots,&
E[\varepsilon_{n}\otimes \varepsilon_{n}]
\end{array}\right]=\left[\begin{array}{ccc}R_{\varepsilon_{1}\varepsilon_{1}}&,\dots,&
R_{\varepsilon_{1}\varepsilon_{n}}\\
\vdots & \vdots & \vdots \\
R_{\varepsilon_{n}\varepsilon_{1}}&,\dots,&
R_{\varepsilon_{n}\varepsilon_{n}}
\end{array}\right],\label{covest2}
\end{eqnarray}

\noindent where $R_{\varepsilon_{i}\varepsilon_{i}}$ $i=1,\dots,n,$ are compact and self-adjoint operators on $H,$ in the trace class.
As before, $\sigma $ represents a scale parameter. In the subsequent development, we assume that $R_{\varepsilon_{i}\varepsilon_{i}}$ $i=1,\dots,n,$ are strictly positive.
 For $i\neq j,$ with $i,j\in \{1,\dots,n\},$ $R_{\varepsilon_{i}\varepsilon_{j}}$ denotes the cross-covariance operator between $\varepsilon_{i}$ and $\varepsilon_{j}.$
Here, $\boldsymbol{\beta
}(\cdot)=[\beta_{1}(\cdot),\dots,\beta_{p}(\cdot)]^{T}\in H^{p},$
 and $\mathbf{X}$ is  a real-valued
$n\times p$ matrix, the fixed effect design matrix.
\begin{remark}
Note that the Gaussian error term of model
(\ref{eq1}) has covariance matrix operator  (\ref{covest2}), which
in the particular case of
$R_{\varepsilon_{i}\varepsilon_{j}}\equiv\mathbf{0},$ for $i\neq j,$
and $i,j\in \{1,\dots,n\},$ and
$R_{\varepsilon_{i}\varepsilon_{i}}=R,$ for $i=1,\dots,n,$ coincides
with the covariance operator (\ref{covest1}) of the error term in
(\ref{fem}), as given  in \cite{Zoglat}. \textcolor{blue}{Thus, equation  (\ref{eq1})
provides an extension  of model (\ref{fem}), going beyond the space   $L^{2}([0,1]),$ and the independence and homoscedastic assumptions.}
\end{remark}
\begin{remark}
Further research can be developed  considering the non-linear and non-gaussian case, and a more flexible class of  design matrices.
 Fixed design  assumption is very common in the
standard multiple (finite-dimensional) regression theory, as well as
it can be justified in the functional regression setting (see, for
example, \cite{Cuevas14}). Partial linear regression is formulated
in the functional setting  in \cite{Aneiros-Perez},  displaying a
better performance
 than   the
functional linear regression and functional nonparametric
regression (see, for example,
\cite{Aneiros-Perez08}; \cite{Zhang12}; \cite{Zhou12}).  Recent advances have extended the linear
approach by combining it with link functions,  considering
multiple indices. But this approach still requires to be improved. The authors in \cite{Chen11}  introduce a new technique for estimating
the link function in a nonparametric framework. An approach
to multi-index modeling using adaptively defined linear projections
of functional data is proposed, which  enables prediction with
polynomial convergence rates. In
\cite{Goia14}, the framework of functional regression
modeling with scalar response is considered. The unknown regression operator  is approximated in a semi-parametric
way  through a single index approach. Possible structural changes are taken into account in the presented approach. In particular, non-smooth functional
directions and additive link functions are used for managing ruptures by applying Single Index Model.
\end{remark}

In the subsequent development, the following assumption is made:

\medskip

\noindent{\bfseries Assumption A0.} The auto-covariance operators
$R_{\varepsilon_{i} \varepsilon_{i}},$ $i=1,\dots n,$ admit a
spectral decomposition in terms of a common complete orthogonal
eigenvector system $\{\phi_{k}\}_{k\geq 1}$ defining in  $H$ the
resolution of the identity $\sum_{k=1}^{\infty }\phi_{k}\otimes
\phi_{k}.$  Let  $\{\eta_{ki},\ k\geq 1\},$ $i=1,\dots,n,$ be the
standard Gaussian random variable sequences such that
$\{\left\langle \varepsilon_{i},
\phi_{k}\right\rangle_{H}=\sqrt{\lambda_{ki}}\eta_{ki},\ k\geq 1\},$
 with $R_{\varepsilon_{i}\varepsilon_{i}}\phi_{k}=\lambda_{ki}\phi_{k},$ for $i=1,\dots,n.$
 The following
orthogonality condition is assumed to be satisfied:
\begin{equation}E[\eta_{ki}\eta_{pj}]=\delta_{k,p},\quad k,p\in \mathbb{N}-\{0\}
\quad i,j=1,\dots,n,\label{eqorth}
\end{equation} \noindent where $\delta $ denotes, as before, the
Kronecker delta function.

\begin{remark}
{\bfseries Assumption A0}  provides a semiparametric definition of the elements
of the  class of covariance matrix  operators characterizing the correlation structure of the
 functional vector error term. Specifically, these elements  admit   an infinite series representation
in terms of a sequence of finite-dimensional matrices
$\{\boldsymbol{\Lambda}_{k},\ k\geq 1\}$ (the parametric part),
introduced in equation (\ref{eqprocovop}) below, with respect to a
resolution of the identity of the Hilbert space $H$ (the non-parametric part), given by  $\{ \phi_{k},\ k\geq 1\}.$
\end{remark}
\bigskip

 Under {\bfseries Assumption A0}, the following  orthogonal
expansions hold for $\varepsilon_{i},$ $i=1,\dots, n,$ in terms of
their common covariance operator eigenvector system $\{
\phi_{k},\ k\geq 1\}:$
\begin{equation}
\varepsilon_{i}=\sum_{k=1}^{\infty
}\sqrt{\lambda_{ki}}\eta_{ki}\phi_{k},\quad i=1,\dots, n.
\label{KLexpansion}
\end{equation}
In addition, from (\ref{KLexpansion}), using the assumed orthogonality condition (\ref{eqorth}) \begin{equation}R_{\varepsilon_{i} \varepsilon_{j}}=E\left[
\varepsilon_{i}\otimes \varepsilon_{j}\right]=\sum_{k=1}^{\infty }\sqrt{\lambda_{ki}\lambda_{kj}}\phi_{k}\otimes \phi_{k},\quad i,j=1,\dots,n.\label{sdcco}
\end{equation}
\noindent  Hence,   the covariance matrix  operator
  (\ref{covest2}) can be rewritten as:
 \begin{equation}\boldsymbol{R}_{\boldsymbol{\varepsilon }
\boldsymbol{\varepsilon }}=\left[\begin{array}{cccc}
\sum_{k=1}^{\infty }\lambda_{k1}\phi_{k}\otimes \phi_{k} & \dots &
\dots & \sum_{k=1}^{\infty
}[\lambda_{k1}\lambda_{kn}]^{1/2}\phi_{k}\otimes \phi_{k}\\
\vdots & \vdots & \vdots & \vdots \\
\vdots & \vdots & \vdots & \vdots \\
\sum_{k=1}^{\infty }[\lambda_{kn}\lambda_{k1}]^{1/2}\phi_{k}\otimes
\phi_{k} & \dots & \dots & \sum_{k=1}^{\infty
}\lambda_{kn}\phi_{k}\otimes \phi_{k} \\
\end{array}
\right].\label{covopsd}
\end{equation}

\subsection{\bfseries Example}
\label{Exm}

An example of   zero-mean Gaussian $H^{n}$-valued  random variable
$\boldsymbol{\varepsilon}$  satisfying {\bfseries Assumption A0} is
now  constructed considering the space $H=L^{2}(\mathbb{R}^{d}).$
Specifically, define $\boldsymbol{\varepsilon}$ as the
$[L^{2}(\mathbb{R}^{d})]^{n}$-valued zero-mean Gaussian
 solution, in the mean-square sense, of the system of stochastic fractional pseudodifferential equations
\begin{equation}f_{i}(\mathcal{L})\varepsilon_{i}=\epsilon_{i},\quad i=1,\dots,n,\label{sdeex}
\end{equation}
\noindent where $\mathcal{L}$ is an elliptic fractional pseudodifferential operator on $L^{2}(\mathbb{R}^{d})$ of  positive order  $s\in \mathbb{R}_{+},$
whose inverse belongs to the Hilbert-Schmidt operator class  on $L^{2}(\mathbb{R}^{d}),$ and  $f_{i},$ $i=1,\dots,n,$ are continuous functions. Here,
$\epsilon_{i},$ $i=1,\dots,n,$ are zero-mean Gaussian $H$-valued  random variables such that $\eta_{ki}=\left\langle \phi_{k},\epsilon_{i}\right\rangle_{L^{2}(\mathbb{R}^{d})},$ $k\geq 1,\quad i=1,\dots,n,$
 satisfying the orthogonality condition (\ref{eqorth}).
From spectral theorems on functional calculus for self-adjoint operators on a Hilbert space $H$ (see, for example,  \cite{Dautray},  pp. 112-126),
the covariance operators $R_{\varepsilon_{i}\varepsilon_{i}}=[f_{i}(\mathcal{L})]^{-1}[[f_{i}(\mathcal{L})]^{-1}]^{*},$ $i=1,\dots,n,$
 with $A^{*}$ denoting the adjoint of operator $A,$ admit a spectral kernel representation in terms of the eigenvectors $\{\phi_{k},\  k \geq 1\}$
of operator $\mathcal{L}$ in the following form: For $i=1,\dots,n,$ and for  $h\in H=L^{2}(\mathbb{R}^{d}),$
\begin{eqnarray}
R_{\varepsilon_{i}\varepsilon_{i}}(h)(\mathbf{x})&=&\sum_{k=1}^{\infty
}\left| f_{i}(\lambda_{k}(\mathcal{L}))\right|^{-2}\left\langle
\phi_{k},h\right\rangle_{L^{2}(\mathbb{R}^{d})}\phi_{k}(\mathbf{x}),\quad
\mathbf{x}\in \mathbb{R}^{d},\label{excov}
\end{eqnarray}
\noindent where $\lambda_{k}(\mathcal{L}),$ $k\geq 1,$ are the
eigenvalues of operator $\mathcal{L}.$ In addition,
under the orthogonality condition (\ref{eqorth}), for  $h\in
H=L^{2}(\mathbb{R}^{d}),$
\begin{equation}
R_{\varepsilon_{i}\varepsilon_{j}}(h)(\mathbf{x})=
\sum_{k=1}^{\infty }\left|
f_{i}(\lambda_{k}(\mathcal{L}))f_{j}(\lambda_{k}(\mathcal{L}))\right|^{-1}
\left\langle
\phi_{k},h\right\rangle_{L^{2}(\mathbb{R}^{d})}\phi_{k}(\mathbf{x}),\quad
\mathbf{x}\in \mathbb{R}^{d}.\label{excov2}
\end{equation}
\vspace*{0.5cm}

Let us now fix some notation and preliminary results for the
subsequent development. By $\mathcal{H}=H^{n}$ we will denote, as
before, the Hilbert space of vector functions in $H^{n}$ with the
inner product $\left\langle \mathbf{f},\mathbf{g}\right\rangle_{\mathcal{H}}=
\sum_{i=1}^{n}\left\langle f_{i},g_{i}\right\rangle_{H},$
for $\mathbf{f}=[f_{1},\dots ,f_{n}]^{T},\mathbf{g}=[g_{1},\dots
,g_{n}]^{T}\in \mathcal{H}.
$
Given an orthonormal system $\{\phi_{k},\ k\geq 1\}$ of $H,$ we
denote by $\boldsymbol{\Phi}^{*}=(\Phi_{k}^{*})_{k\geq 1},$ with
 $\Phi_{k}^{*}:
H^{n}\longrightarrow \mathbb{R}^{n},$
the
projection operator into such a system of the components of any
vector function $\mathbf{f}\in H^{n}$ as follows:
\begin{eqnarray}
\boldsymbol{\Phi}^{*}(\mathbf{f})&=&(\Phi_{k}^{*}(\mathbf{f}))_{k\geq
1}=(\mathbf{f}_{k})_{k\geq
1}\nonumber\\&=& \left(\left( \left\langle
f_{1},\phi_{k}\right\rangle_{H},\dots, \left\langle
f_{n},\phi_{k}\right\rangle_{H} \right)^{T}\right)_{k\geq
1}=\left(\left(f_{k1},\dots,f_{kn}\right)^{T}\right)_{k\geq 1}.
\label{eqproj1}
\end{eqnarray}
 The inverse operator
$\boldsymbol{\Phi}$  of $\boldsymbol{\Phi}^{*}$ satisfies
$\boldsymbol{\Phi}\boldsymbol{\Phi}^{*}=\mathbf{I}_{H^{n}},$ i.e.,
$\boldsymbol{\Phi}: [l^{2}]^{n}\longrightarrow H^{n},$ is given by

\begin{equation}\boldsymbol{\Phi}\left(\left(\left(f_{k1},\dots,f_{kn}\right)^{T}\right)_{k\geq
1}\right)=\left( \sum_{k=1}^{\infty
}f_{k1}\phi_{k},\dots,\sum_{k=1}^{\infty
}f_{kn}\phi_{k}\right)^{T},\label{eqproj2b}
\end{equation}
\noindent for all
$\left(\left(f_{k1},\dots,f_{kn}\right)^{T}\right)_{k\geq 1}\in
[l^{2}]^{n}.$ Let us consider a  matrix operator $\boldsymbol{\mathcal{A}}$ with
functional entries defined by the operators
$A_{i,j}=\sum_{k=1}^{\infty }\gamma_{kij}\phi_{k}\otimes \phi_{k},$
 with $\sum_{k=1}^{\infty
}\gamma_{kij}^{2}<\infty,$  and
$A_{i,j}(f)(g)=\sum_{k=1}^{\infty }\gamma_{kij}\left\langle
f,\phi_{k}\right\rangle_{H}\left\langle
g,\phi_{k}\right\rangle_{H},\quad i,j=1,\dots,n,$ for $f,g\in
H.$ Then, the following notation will  also be used
\begin{eqnarray}
\boldsymbol{\Phi}^{*}\boldsymbol{\mathcal{A}}\boldsymbol{\Phi}&=&
\left(\left[\begin{array}{ccc}\gamma_{k11} &\dots
 &\gamma_{k1n}\\
\dots & \dots & \dots \\
\gamma_{kn1} &\dots & \gamma_{knn}
\end{array}\right]\right)_{k\geq 1}.
\label{eqmpn}
\end{eqnarray}
\noindent Reciprocally,  given an infinite sequence of $n\times n$
matrices $\left(\left[\begin{array}{ccc}\gamma_{k11} &\dots
 &\gamma_{k1n}\\
\dots & \dots & \dots \\
\gamma_{kn1} &\dots & \gamma_{knn}
\end{array}\right]\right)_{k\geq 1},$
\begin{eqnarray}
& & \boldsymbol{\Phi}\left(\left[\begin{array}{ccc}\gamma_{k11}
&\dots
 &\gamma_{k1n}\\
\dots & \dots & \dots \\
\gamma_{kn1} &\dots & \gamma_{knn}
\end{array}\right]\right)_{k\geq 1}\boldsymbol{\Phi}^{*}\nonumber\\
&& \nonumber\\ & &=\left[\begin{array}{ccc}\sum_{k=1}^{\infty
}\gamma_{k11}\phi_{k}\otimes \phi_{k} &\dots
 &\sum_{k=1}^{\infty
}\gamma_{k1n}\phi_{k}\otimes \phi_{k}\\
\dots & \dots & \dots \\
\sum_{k=1}^{\infty }\gamma_{kn1}\phi_{k}\otimes \phi_{k} &\dots &
\sum_{k=1}^{\infty }\gamma_{knn}\phi_{k}\otimes \phi_{k}
\end{array}\right].
\label{eqmpn2}
\end{eqnarray}
 Thus, from  (\ref{covopsd}), considering
 (\ref{eqmpn})--(\ref{eqmpn2}),
\begin{equation}\boldsymbol{\Phi}^{*}\boldsymbol{R}_{\boldsymbol{\varepsilon }
\boldsymbol{\varepsilon }}\boldsymbol{\Phi}=
\left(\left[\begin{array}{ccc}\lambda_{k1} &\dots
 &\sqrt{\lambda_{k1}\lambda_{kn}}\\
\dots & \dots & \dots \\
\sqrt{\lambda_{kn}\lambda_{k1}}&\dots & \lambda_{kn}
\end{array}\right]\right)_{k\geq 1}=\left(\boldsymbol{\Lambda
}_{k}\right)_{k\geq 1}, \label{eqprocovop}
\end{equation}
\noindent and from equations (\ref{eqproj1}) and (\ref{eqproj2b}),
\begin{eqnarray}
\boldsymbol{R}_{\boldsymbol{\varepsilon } \boldsymbol{\varepsilon
}}(\mathbf{f})(\mathbf{g})&=&
\boldsymbol{\Phi}\boldsymbol{\Phi}^{*}\boldsymbol{R}_{\boldsymbol{\varepsilon
} \boldsymbol{\varepsilon
}}\boldsymbol{\Phi}(\boldsymbol{\Phi}^{*}\mathbf{f})(\mathbf{g})
=\boldsymbol{\Phi}^{*}\boldsymbol{R}_{\boldsymbol{\varepsilon }
\boldsymbol{\varepsilon }}\boldsymbol{\Phi}(\boldsymbol{\Phi}^{*}
\mathbf{f})(\boldsymbol{\Phi}^{*}\mathbf{g})
\nonumber\\
&=&\sum_{k=1}^{\infty }[\left\langle
g_{1},\phi_{k}\right\rangle_{H},\dots,\left\langle
g_{n},\phi_{k}\right\rangle_{H}] \boldsymbol{\Lambda
}_{k}[\left\langle
f_{1},\phi_{k}\right\rangle_{H},\dots,\left\langle
f_{n},\phi_{k}\right\rangle_{H}]^{T}
\nonumber\\
&=&\sum_{k=1}^{\infty }\mathbf{g}_{k}^{T}\boldsymbol{\Lambda }_{k}
\mathbf{f}_{k}= \left\langle \boldsymbol{\Phi}\boldsymbol{\Lambda
}^{1/2}\boldsymbol{\Phi}^{*}(\mathbf{f}),\boldsymbol{\Phi}\boldsymbol{\Lambda
}^{1/2}\boldsymbol{\Phi}^{*}(\mathbf{g})\right\rangle_{H^{n}},
\label{covopme2}
\end{eqnarray}
\noindent for all  $\mathbf{f},$ $\mathbf{g} \in \mathcal{H}=H^{n},$
with $\boldsymbol{\Lambda }^{1/2}:=(\boldsymbol{\Lambda
}^{1/2}_{k})_{k\geq 1},$ $\mathbf{g}_{k}=(g_{k1},\dots, g_{kn})^{T}$
and $\mathbf{f}_{k}=(f_{k1},\dots, f_{kn})^{T},$ for every $k\geq
1.$ Thus, $\boldsymbol{\Phi}^{*}\boldsymbol{R}_{\boldsymbol{\varepsilon } \boldsymbol{\varepsilon
}}^{1/2}\boldsymbol{\Phi}=\left(\boldsymbol{\Lambda
}^{1/2}_{k}\right)_{k\geq 1}=\boldsymbol{\Lambda }^{1/2}.$
Note that from Cauchy-Schwarz inequality,
\begin{equation}
\sum_{k=1}^{\infty }[\lambda_{ki}\lambda_{kj}]^{1/2}\leq
\left[\sum_{k=1}^{\infty
}\lambda_{ki}\right]^{1/2}\left[\sum_{k=1}^{\infty
}\lambda_{kj}\right]^{1/2}<\infty,\quad i,j=1,\dots,n,
\label{SI}
\end{equation}
\noindent since $E\|\varepsilon_{i}\|_{H}^{2}=\sum_{k=1}^{\infty
}\lambda_{ki}<\infty,$  for $i=1,\dots,n,$ and hence,
$R_{\varepsilon_{i}\varepsilon_{i}},$ $i=1,\dots,n,$ are  positive
self-adjoint trace covariance operators, i.e., $\lambda_{ki}> 0,$
$k\geq 1,$ and $\sum_{k=1}^{\infty }\lambda_{ki}<\infty,$ for
$i=1,\dots,n.$ Consequently,
 \begin{equation}\mbox{trace}\left(\sum_{k=1}^{\infty
}\boldsymbol{\Lambda }_{k}\right)= \sum_{k=1}^{\infty
}\mbox{trace}\left(\boldsymbol{\Lambda
}_{k}\right)<\infty.\label{tracecvo}
\end{equation}
From equation (\ref{covopme2}), for every
$\mathbf{f}=(f_{1},\dots,f_{n})^{T}\in \mathcal{H}=H^{n},$
\begin{equation}\mathbf{g}(\cdot)=\boldsymbol{\Phi}\left(
\left(\boldsymbol{\Lambda }_{k}^{1/2}\mathbf{f}_{k}\right)_{k\geq
1}\right)\in \boldsymbol{R}_{\boldsymbol{\varepsilon }
\boldsymbol{\varepsilon }}^{1/2}(\mathcal{H}).\label{RKHSprocc}
\end{equation}
\noindent  Hence, $\boldsymbol{\Phi}^{*}(\mathbf{g})=
\left(\boldsymbol{\Lambda }_{k}^{1/2}\mathbf{f}_{k}\right)_{k\geq
1},$ and
\begin{equation}
\boldsymbol{\Phi}^{*}\boldsymbol{\mathcal{Q}}\boldsymbol{\Phi}=\left(\boldsymbol{\Lambda
}_{k}^{-1/2}\right)_{k\geq 1}:
\boldsymbol{R}_{\boldsymbol{\varepsilon } \boldsymbol{\varepsilon
}}^{1/2}(\mathcal{H})=\mathcal{H}(\boldsymbol{\varepsilon})\longrightarrow \mathcal{H},\ \boldsymbol{\Phi}^{*}\boldsymbol{\mathcal{Q}}\mathbf{g}= \left(\mathbf{f}_{k}\right)_{k\geq
1}.\label{isqr}
\end{equation}
\begin{lemma} \label{lem10} The inverse $\boldsymbol{R}_{\boldsymbol{\varepsilon }
\boldsymbol{\varepsilon }}^{-1}$ of the matrix covariance
operator satisfies
\begin{equation} \boldsymbol{R}_{\boldsymbol{\varepsilon } \boldsymbol{\varepsilon
}}^{-1}(\boldsymbol{\psi})(\boldsymbol{\varphi})=\sum_{k=1}^{\infty
} \boldsymbol{\varphi}_{k}^{T} \boldsymbol{\Lambda
}_{k}^{-1}\boldsymbol{\psi}_{k}, \label{covopme2bb}
\end{equation}
\noindent for all $\boldsymbol{\psi}, \boldsymbol{\varphi}\in
\boldsymbol{R}_{\boldsymbol{\varepsilon } \boldsymbol{\varepsilon
}}^{1/2}(\mathcal{H}).$
 Equivalently,
$\boldsymbol{R}_{\boldsymbol{\varepsilon } \boldsymbol{\varepsilon
}}^{-1}$ is such that
\begin{equation}\boldsymbol{\Phi}^{*}\boldsymbol{R}_{\boldsymbol{\varepsilon } \boldsymbol{\varepsilon
}}^{-1}\boldsymbol{\Phi}=\left(\boldsymbol{\Lambda
}_{k}^{-1}\right)_{k\geq 1}.\label{sil1}
\end{equation}
\end{lemma}
\begin{proof}
From equations (\ref{covopme2}) and (\ref{RKHSprocc}),
$\boldsymbol{R}_{\boldsymbol{\varepsilon } \boldsymbol{\varepsilon
}}^{1/2}(\mathcal{H})=\boldsymbol{\Phi}\boldsymbol{\Lambda
}^{1/2}(\boldsymbol{\Phi}^{*}(H^{n})).$
In particular, for every $\boldsymbol{\psi}, \boldsymbol{\varphi}\in
\boldsymbol{R}_{\boldsymbol{\varepsilon } \boldsymbol{\varepsilon
}}^{1/2}(\mathcal{H}),$ there exist $\mathbf{f}$ and $\mathbf{g}\in
\mathcal{H}=H^{n}$ such that
\begin{eqnarray}
\boldsymbol{\psi}&=&\boldsymbol{\Phi}\boldsymbol{\Lambda
}^{1/2}\boldsymbol{\Phi}^{*}\mathbf{f}\nonumber\\
\boldsymbol{\varphi}&=&\boldsymbol{\Phi}\boldsymbol{\Lambda
}^{1/2}\boldsymbol{\Phi}^{*}\mathbf{g}.\label{RKHS}
\end{eqnarray}
From equations (\ref{covopme2}), (\ref{isqr}) and (\ref{RKHS})
\begin{eqnarray}
\boldsymbol{\mathcal{Q}}^{*}\boldsymbol{\mathcal{Q}}(\boldsymbol{\psi})(\boldsymbol{\varphi})&=&
\left\langle
\boldsymbol{\mathcal{Q}}(\boldsymbol{\psi}),\boldsymbol{\mathcal{Q}}(\boldsymbol{\varphi})\right\rangle_{H^{n}}
\nonumber\\
&=&\sum_{k=1}^{\infty }\boldsymbol{\psi}_{k}^{T}\boldsymbol{\Lambda
}^{-1}_{k}\boldsymbol{\varphi}_{k} = \sum_{k=1}^{\infty
}\mathbf{f}_{k}^{T}\boldsymbol{\Lambda
}^{1/2}_{k}\boldsymbol{\Lambda
}^{-1}_{k}\boldsymbol{\Lambda }^{1/2}_{k}\mathbf{g}_{k}\nonumber\\
&=& \sum_{k=1}^{\infty
}\mathbf{f}_{k}^{T}\mathbf{g}_{k}=\left\langle
\mathbf{f},\mathbf{g}\right\rangle_{H^{n}}=\mathbf{I}_{H^{n}}(\mathbf{f})(\mathbf{g})=\boldsymbol{R}_{\boldsymbol{\varepsilon
} \boldsymbol{\varepsilon
}}^{-1}(\boldsymbol{\psi})(\boldsymbol{\varphi}),
 \label{RKHS2}
\end{eqnarray}
\noindent where $\mathbf{I}_{H^{n}}$ denotes the identity operator
on $H^{n}.$ From (\ref{RKHS}) and (\ref{RKHS2}), equations
(\ref{covopme2bb}) and (\ref{sil1}) are obtained.
\end{proof}

\begin{remark}
\label{kev}
In the following development we will  assume that the
eigenvector system  $\{\phi_{k},\ k\geq1\}$  and the matrix sequence
 $\{ \boldsymbol{\Lambda}_{k},\ k\geq 1\}$  are known. Thus, we address the problem of least-squares estimation of $H^{p}$-valued parameter $\boldsymbol{\beta},$ under suitable conditions
 on the fixed design matrix $\mathbf{X}$ (see Section \ref{sec3}).
  In the case where $\boldsymbol{\varepsilon }$ is defined by  a system of stochastic partial differential  or  pseudodifferential equations (see  Subsection \ref{Exm}),  equations (\ref{excov})--(\ref{excov2}) show that the functional entries of the covariance matrix operator    (\ref{covest2}) are given in terms of the  eigenvector system of operator $\mathcal{L}$ (which is known), and  the   eigenvalues of $\mathcal{L},$ transformed
  by the   continuous functions $f_{i},$ $i=1,\dots,n,$ defining the system of stochastic differential or  pseudodifferential equations (\ref{sdeex}).
Covariance operator estimation has been addressed in \cite{Bosq}, \cite{Horvath} and \cite{Kokoszka}.
\end{remark}
\section{Generalized least-squares estimation in the
RKHS norm}

\label{sec3} It is well-known (see, for example, \cite{Da Prato},
pp. 12-13) that the RKHS $\mathcal{H}(\boldsymbol{\varepsilon})$ of
$\boldsymbol{\varepsilon}$ is defined as the closure of
$R_{\boldsymbol{\varepsilon}\boldsymbol{\varepsilon}}^{1/2}(\mathcal{H})$
in the norm
$\|\cdot\|_{R_{\boldsymbol{\varepsilon}\boldsymbol{\varepsilon}}^{-1}}$
induced by
$R_{\boldsymbol{\varepsilon}\boldsymbol{\varepsilon}}^{-1}.$
 From Lemma \ref{lem10},
\begin{eqnarray} & &
\| \mathbf{Y}-\mathbf{X}\boldsymbol{\beta
}\|_{\mathbf{R}_{\boldsymbol{\varepsilon}\boldsymbol{\varepsilon}}^{-1}}^{2}=
\mathbf{R}^{-1}_{\boldsymbol{\varepsilon}\boldsymbol{\varepsilon}}(\mathbf{Y}-\mathbf{X}\boldsymbol{\beta
})(\mathbf{Y}-\mathbf{X}\boldsymbol{\beta
})\nonumber\\
& &=
\boldsymbol{\Phi}^{*}\mathbf{R}^{-1}_{\boldsymbol{\varepsilon}\boldsymbol{\varepsilon}}\boldsymbol{\Phi}(\boldsymbol{\Phi}^{*}(\mathbf{Y}-\mathbf{X}\boldsymbol{\beta
}))(\boldsymbol{\Phi}^{*}(\mathbf{Y}-\mathbf{X}\boldsymbol{\beta }))
\nonumber\\
& & =\sum_{k=1}^{\infty }
\left[\mathbf{Y}_{k}-[\mathbf{X}\boldsymbol{\beta
}]_{k}\right]^{T}\boldsymbol{\Lambda
}_{k}^{-1}\left[\mathbf{Y}_{k}-[\mathbf{X}\boldsymbol{\beta
}]_{k}\right] =\sum_{k=1}^{\infty }\| \boldsymbol{\varepsilon
}_{k}(\boldsymbol{\beta }_{k})\|^{2}_{\boldsymbol{\Lambda
}_{k}^{-1}}, \label{eqrnormcn}
\end{eqnarray} \noindent where, as before
$\boldsymbol{\varepsilon}=\mathbf{Y}-\mathbf{X}\boldsymbol{\beta },$
and $\boldsymbol{\varepsilon}_{k}(\boldsymbol{\beta }_{k})=\left[\mathbf{Y}_{k}-[\mathbf{X}\boldsymbol{\beta
}]_{k}\right]=\Phi_{k}^{*}(\mathbf{Y}-\mathbf{X}\boldsymbol{\beta
}),$ $k\geq 1.$ Equation (\ref{eqrnormcn}) is minimized if and only if, for each
$k\geq 1,$ the  norm $\|\cdot \|_{\boldsymbol{\Lambda }_{k}^{-1}}$
of $\varepsilon_{k}(\boldsymbol{\beta }_{k})$ is minimized.
For each $k\geq 1,$ the minimizer of $\| \boldsymbol{\varepsilon
}_{k}(\boldsymbol{\beta }_{k})\|^{2}_{\boldsymbol{\Lambda
}_{k}^{-1}}$ with respect to $\boldsymbol{\beta }_{k}$ is given by
the generalized least squares estimator
\begin{equation}\widehat{\boldsymbol{\beta }}_{k}=(\widehat{\beta}_{k1},\dots,\widehat{\beta}_{kp})^{T}=(\mathbf{X}^{T}\boldsymbol{\Lambda }_{k}^{-1}
\mathbf{X})^{-1}\mathbf{X}^{T}\boldsymbol{\Lambda
}_{k}^{-1}\mathbf{Y}_{k},\quad k\geq 1.\label{lse1} \end{equation}
Hence, the corresponding approximation of the functional
vector parameter $\boldsymbol{\beta }$ is obtained from (\ref{lse1}) by
considering
\begin{equation}\widehat{\boldsymbol{\beta }}=\boldsymbol{\Phi}((\widehat{\boldsymbol{\beta }}_{k})_{k\geq
1})=\left(\sum_{k=1}^{\infty }\widehat{\beta}_{k1}\phi_{k},\dots,
\sum_{k=1}^{\infty
}\widehat{\beta}_{kp}\phi_{k}\right)^{T}.\label{solutlsep}
\end{equation}
Under {\bfseries Assumption A0}, equations
(\ref{eqrnormcn})--(\ref{solutlsep}) provide a statistics minimizing
the functional mean-square error, computed in the RKHS norm of
$\mathcal{H}(\boldsymbol{\varepsilon}).$  This statistics will
define an  estimator of parameter $\boldsymbol{\beta },$ hence, a
generalized least squares estimator, if the following condition is
satisfied:
\begin{equation}
\sum_{k=1}^{\infty}\sum_{i=1}^{p}\widehat{\beta}_{ki}^{2}=
\sum_{k=1}^{\infty} \left[(\mathbf{X}^{T}\boldsymbol{\Lambda
}_{k}^{-1} \mathbf{X})^{-1}\mathbf{X}^{T}\boldsymbol{\Lambda
}_{k}^{-1}\mathbf{Y}_{k}\right]^{T}\left[(\mathbf{X}^{T}\boldsymbol{\Lambda
}_{k}^{-1} \mathbf{X})^{-1}\mathbf{X}^{T}\boldsymbol{\Lambda
}_{k}^{-1}\mathbf{Y}_{k}\right]<\infty , \label{eqsnc}
\end{equation}
\noindent i.e., if $\widehat{\boldsymbol{\beta }}\in H^{p},$ with
$H^{p}$ being the functional parameter space where
$\boldsymbol{\beta }$
 takes its values. The following proposition provides a sufficient condition that ensures that the statistics given in  (\ref{solutlsep}) is a generalized least-squares
 estimator for functional parameter vector $\boldsymbol{\beta }.$
\begin{proposition}
\label{pr1} If
\begin{equation}\sum_{k=1}^{\infty
}\mathrm{trace}\left(\mathbf{X}^{T}\boldsymbol{\Lambda }_{k}^{-1}
\mathbf{X}\right)^{-1}<\infty,\label{dmsc}
\end{equation}\noindent  then,
 equation (\ref{eqsnc})
is satisfied a.s. Consequently, $\widehat{\boldsymbol{\beta }}$ in
equation (\ref{solutlsep}) defines a generalized least-squares
 estimator for $\boldsymbol{\beta }.$
\end{proposition}
\begin{proof}
Under condition (\ref{dmsc}),  since $\|\boldsymbol{\beta
}\|^{2}_{H^{p}}=\sum_{k=1}^{\infty}
\boldsymbol{\beta}_{k}^{T}\boldsymbol{\beta}_{k}<\infty,$
 we obtain
\begin{eqnarray}
& &
E\left(\sum_{k=1}^{\infty}\sum_{i=1}^{p}\widehat{\beta}_{ki}^{2}\right)=
\sum_{k=1}^{\infty} E\left(\mathbf{Y}_{k}^{T}\boldsymbol{\Lambda
}_{k}^{-1}\mathbf{X}(\mathbf{X}^{T}\boldsymbol{\Lambda }_{k}^{-1}
\mathbf{X})^{-1}(\mathbf{X}^{T}\boldsymbol{\Lambda }_{k}^{-1}
\mathbf{X})^{-1}\mathbf{X}^{T}\boldsymbol{\Lambda
}_{k}^{-1}\mathbf{Y}_{k} \right)\nonumber\\
& &=\sum_{k=1}^{\infty}\mbox{trace}\left(\boldsymbol{\Lambda
}_{k}^{-1}\mathbf{X}(\mathbf{X}^{T}\boldsymbol{\Lambda }_{k}^{-1}
\mathbf{X})^{-1}(\mathbf{X}^{T}\boldsymbol{\Lambda }_{k}^{-1}
\mathbf{X})^{-1}\mathbf{X}^{T}\boldsymbol{\Lambda
}_{k}^{-1}\boldsymbol{\Lambda }_{k}\right)\nonumber\\
& & +\sum_{k=1}^{\infty}\boldsymbol{\beta
}^{T}_{k}\mathbf{X}^{T}\boldsymbol{\Lambda
}_{k}^{-1}\mathbf{X}(\mathbf{X}^{T}\boldsymbol{\Lambda }_{k}^{-1}
\mathbf{X})^{-1}(\mathbf{X}^{T}\boldsymbol{\Lambda }_{k}^{-1}
\mathbf{X})^{-1}\mathbf{X}^{T}\boldsymbol{\Lambda
}_{k}^{-1}\mathbf{X}\boldsymbol{\beta }_{k}\nonumber
\end{eqnarray}
\begin{eqnarray}
& &= \sum_{k=1}^{\infty}\mbox{trace}\left(\boldsymbol{\Lambda
}_{k}^{-1}\mathbf{X}(\mathbf{X}^{T}\boldsymbol{\Lambda }_{k}^{-1}
\mathbf{X})^{-1}(\mathbf{X}^{T}\boldsymbol{\Lambda }_{k}^{-1}
\mathbf{X})^{-1}\mathbf{X}^{T}\right)+\boldsymbol{\beta
}^{T}_{k}\boldsymbol{\beta }_{k}
\nonumber\\
& &=
\sum_{k=1}^{\infty}\mbox{trace}\left(\mathbf{X}^{T}\boldsymbol{\Lambda
}_{k}^{-1}\mathbf{X}(\mathbf{X}^{T}\boldsymbol{\Lambda }_{k}^{-1}
\mathbf{X})^{-1}(\mathbf{X}^{T}\boldsymbol{\Lambda }_{k}^{-1}
\mathbf{X})^{-1}\right)+\| \boldsymbol{\beta }\|^{2}_{H^{p}}
\nonumber\\
& &
=\sum_{k=1}^{\infty}\mbox{trace}(\mathbf{X}^{T}\boldsymbol{\Lambda
}_{k}^{-1} \mathbf{X})^{-1}+\| \boldsymbol{\beta }\|^{2}_{H^{p}}
<\infty,
 \label{proofpr11}
\end{eqnarray}
\noindent where we have used the well-known  formula
\begin{equation}E[\mathbf{y}^{T}\mathbf{Q}\mathbf{y}]=\mbox{trace}(\mathbf{Q}\mathbf{V})+\boldsymbol{\mu
}^{T}\mathbf{Q}\boldsymbol{\mu },\label{eqf} \end{equation}
\noindent for a given symmetric matrix $\mathbf{Q},$ but not
necessarily positive definite, with $\mathbf{y}$ being a random
vector whose mean is $\boldsymbol{\mu },$ and whose
variance-covariance matrix is
$\mathbf{V}=E[\mathbf{y}\mathbf{y}^{T}]-\boldsymbol{\mu
}\boldsymbol{\mu }^{T}$ (see, for example, \cite{Hocking},
\cite{Schott}). Therefore,
$\sum_{k=1}^{\infty}\sum_{i=1}^{p}\widehat{\beta}_{ki}^{2}<\infty$
a.s., and in particular $\widehat{\boldsymbol{\beta }}\in H^{p}$
a.s.
\end{proof}
\begin{remark}
The sufficient condition (\ref{dmsc})  restricts our class of fixed
effect  design matrices to those ones preserving the trace property
(\ref{tracecvo}) of the infinite series of matrices $\{
\Lambda_{k},\ k\geq 1\}.$ In particular, in the case where
$\mathbf{X}$ is the identity matrix $\sum_{k=1}^{\infty}\mathrm{trace}\left(\mathbf{X}^{T}\boldsymbol{\Lambda
}_{k}^{-1}
\mathbf{X}\right)^{-1}=\sum_{k=1}^{\infty}\mathrm{trace}(\boldsymbol{\Lambda
}_{k})<\infty.$

Also, if $\mathbf{X}$ is a unitary matrix such that $\mathbf{X}\mathbf{X}^{T}=\mathbf{I},$ then, condition  (\ref{dmsc}) assumed in Proposition \ref{pr1} is also satisfied, since
$$\sum_{k=1}^{\infty}\mathrm{trace}\left(\mathbf{X}^{T}\boldsymbol{\Lambda }_{k}^{-1}
\mathbf{X}\right)^{-1}=\sum_{k=1}^{\infty}\mathrm{trace}\left(\mathbf{X}\mathbf{X}^{T}\boldsymbol{\Lambda }_{k}^{-1}\right)^{-1}=
\sum_{k=1}^{\infty}\mathrm{trace}\left(\boldsymbol{\Lambda }_{k}\right)<\infty.$$
\end{remark}
\section{Functional analysis of   variance for the
transformed data model}
\label{sec4} Let us first compute  the residual
error sum of squares \textbf{SSE}. From definition of
$\widehat{\boldsymbol{\beta }}$ in equations (\ref{lse1}) and
(\ref{solutlsep}), we have
\begin{eqnarray}
\mathbf{Y}-\mathbf{X}\widehat{\boldsymbol{\beta
}}&=&\boldsymbol{\Phi}\left(\left(\mathbf{Y}_{k}-\mathbf{X}(\mathbf{X}^{T}\boldsymbol{\Lambda
}_{k}^{-1}\mathbf{X})^{-1}\mathbf{X}^{T}\boldsymbol{\Lambda
}_{k}^{-1}\mathbf{Y}_{k})\right)_{k\geq 1}\right)\nonumber\\
&=&\boldsymbol{\Phi}\left(\left( \left(\mathbf{I}_{n\times
n}-\mathbf{X}(\mathbf{X}^{T}\boldsymbol{\Lambda
}_{k}^{-1}\mathbf{X})^{-1}\mathbf{X}^{T}\boldsymbol{\Lambda
}_{k}^{-1}\right)\mathbf{Y}_{k}\right)_{k\geq 1}\right)\nonumber\\
&=&\boldsymbol{\Phi}\left(\left(
\boldsymbol{\mathcal{M}}_{k}\mathbf{Y}_{k}\right)_{k\geq 1}\right)\label{hte}\\
&=& \left(\sum_{k=1}^{\infty
}\left[\sum_{i=1}^{n}\boldsymbol{\mathcal{M}}_{k}(1,i)\mathbf{Y}_{ki}\right]\phi_{k},\dots,
\sum_{k=1}^{\infty
}\left[\sum_{i=1}^{n}\boldsymbol{\mathcal{M}}_{k}(n,i)\mathbf{Y}_{ki}\right]\phi_{k}\right)^{T},
\nonumber
\end{eqnarray}
\noindent where, as before, $\mathbf{I}_{n\times n}$ is the $n\times
n$ identity matrix.

 From Lemma
\ref{lem10} and  (\ref{hte}), the residual error sum of squares
(\textbf{SSE}) is   computed in the  geometry of the RKHS of
$\boldsymbol{\varepsilon}$ as follows:
\begin{eqnarray}
\mbox{\textbf{SSE}}&=&\left\langle
\mathbf{Y}-\widehat{\boldsymbol{Y}},\mathbf{Y}-\widehat{\boldsymbol{Y}}\right\rangle_{R_{\boldsymbol{\varepsilon}\boldsymbol{\varepsilon}}^{-1}}=
\left\langle \mathbf{Y}-\mathbf{X}\widehat{\boldsymbol{\beta
}},\mathbf{Y}-\mathbf{X}\widehat{\boldsymbol{\beta
}}\right\rangle_{R_{\boldsymbol{\varepsilon}\boldsymbol{\varepsilon}}^{-1}}\nonumber\\
&=&
R_{\boldsymbol{\varepsilon}\boldsymbol{\varepsilon}}^{-1}\left(\mathbf{Y}-\mathbf{X}\widehat{\boldsymbol{\beta
}}\right)\left(\mathbf{Y}-\mathbf{X}\widehat{\boldsymbol{\beta
}}\right)\nonumber\\
&=&\sum_{k=1}^{\infty}
[\boldsymbol{\mathcal{M}}_{k}\mathbf{Y}_{k}]^{T}\boldsymbol{\Lambda
}_{k}^{-1}\boldsymbol{\mathcal{M}}_{k}\mathbf{Y}_{k}, \label{SSE2b}
\end{eqnarray}
\noindent where, for each $k\geq 1,$ $\boldsymbol{\mathcal{M}}_{k}$
is given in equation (\ref{hte}).
  Note that $\left\langle
\mathbf{Y}-\mathbf{X}\boldsymbol{\beta },
\mathbf{Y}-\mathbf{X}\boldsymbol{\beta
}\right\rangle_{R_{\boldsymbol{\varepsilon}\boldsymbol{\varepsilon}}^{-1}}\sim
\sum_{k=1}^{\infty}X_{k},\quad X_{k}\sim \chi^{2}(n),$ $k\geq 1.$
In addition,  since, under {\bfseries Assumption A0}, from  Lemma
\ref{lem10}, the total sum of squares ($\mbox{\textbf{SST}}$) is
given by
\begin{equation}
\mbox{\textbf{SST}}=\left\langle \mathbf{Y},
\mathbf{Y}\right\rangle_{R_{\boldsymbol{\varepsilon}\boldsymbol{\varepsilon}}^{-1}}=
R_{\boldsymbol{\varepsilon}\boldsymbol{\varepsilon}}^{-1}\left(
\mathbf{Y}\right)\left( \mathbf{Y}\right)= \sum_{k=1}^{\infty
}\mathbf{Y}_{k}^{T}\boldsymbol{\Lambda}_{k}^{-1}\mathbf{Y}_{k},\label{SST}
\end{equation}
\noindent it follows  that the first two  moments of $\mathbf{SST}$ are not finite, i.e.,
\begin{eqnarray}
E[\mbox{\textbf{SST}}] &=&\sum_{k=1}^{\infty }
\mbox{trace}\left(\boldsymbol{\Lambda}_{k}^{-1}\boldsymbol{\Lambda}_{k}\right)+\boldsymbol{\beta
}^{T}\mathbf{X}^{T}\boldsymbol{\Lambda}_{k}^{-1}\mathbf{X}\boldsymbol{\beta
}=\infty\nonumber\\
\mbox{Var}\left( \mbox{\textbf{SST}}\right) &=&\sum_{k=1}^{\infty
}2\mbox{trace}\left(\boldsymbol{\Lambda}_{k}^{-1}\boldsymbol{\Lambda}_{k}\boldsymbol{\Lambda}_{k}^{-1}\boldsymbol{\Lambda}_{k}\right)+4\boldsymbol{\beta
}^{T}\mathbf{X}^{T}\boldsymbol{\Lambda}_{k}^{-1}\boldsymbol{\Lambda}_{k}\boldsymbol{\Lambda}_{k}^{-1}\mathbf{X}\boldsymbol{\beta
}=\infty.\nonumber
\end{eqnarray}
Hence, we consider a linear transformation of the functional vector
$\mathbf{Y}$ in equation (\ref{eq1}) to ensure that the
corresponding total sum of squares
($\widetilde{\mbox{\textbf{SST}}}$) is almost surely finite. Denote
by $\mathbf{W}:H^{n}\longrightarrow H^{n}$ such a transformation,
thus
\begin{equation}
\widetilde{\mathbf{Y}}=\mathbf{W}\mathbf{Y}(\cdot)=\mathbf{W}\mathbf{X}\boldsymbol{\beta
}(\cdot)+\mathbf{W}\boldsymbol{\varepsilon}(\cdot). \label{tdm}
\end{equation}

\subsection{Conditions on $\mathbf{W}$ for a.s. finiteness of functional variance components}
\label{secW}
 In the  construction of the functional entries of $\mathbf{W},$ we
consider the resolution of the identity in $H$ given by
$\sum_{k=1}^{\infty }\phi_{k}\otimes \phi_{k},$ with $\{ \phi_{k},\
k\geq 1\}$ denoting, as before, the common eigenvector system of the
covariance operators $R_{\varepsilon_{i}\varepsilon_{j}},$
$i,j=1,\dots,n.$ Note that such eigenvectors are assumed to be known
(see Remark \ref{kev}). Specifically, the weight matrix operator
$\mathbf{W}$ will be of the form
\begin{eqnarray}
\mathbf{W}&=&\left[\begin{array}{cccc} \sum_{k=1}^{\infty
}w_{k11}\phi_{k}\otimes \phi_{k} & \dots & \dots &
\sum_{k=1}^{\infty
}w_{k1n}\phi_{k}\otimes \phi_{k}\\
\vdots & \vdots & \vdots & \vdots \\
\vdots & \vdots & \vdots & \vdots \\
\sum_{k=1}^{\infty }w_{kn1}\phi_{k}\otimes \phi_{k} & \dots & \dots
& \sum_{k=1}^{\infty
}w_{knn}\phi_{k}\otimes \phi_{k} \\
\end{array}
\right]. \label{weightop}
\end{eqnarray}
\noindent For the almost surely finiteness of the  total  sum of squares from equation (\ref{tdm}),
$\mathbf{W}$ must satisfy
\begin{eqnarray}
& & \sum_{k=1}^{\infty
}\mbox{trace}\left(\boldsymbol{\Lambda}_{k}\mathbf{W}_{k}^{T}\boldsymbol{\Lambda}_{k}^{-1}\mathbf{W}_{k}\right)<
\infty\nonumber\\
& & \sum_{k=1}^{\infty }\boldsymbol{\beta
}^{T}_{k}\mathbf{X}^{T}\mathbf{W}_{k}^{T}\boldsymbol{\Lambda}_{k}^{-1}\mathbf{W}_{k}\mathbf{X}\boldsymbol{\beta
}_{k}< \infty, \label{eqfiniteness}
\end{eqnarray}
\noindent with
$\Phi_{k}^{*}\mathbf{W}\Phi_{k}=\mathbf{W}_{k}=(w_{kij})_{i,j=1,\dots,n},$
for each $k\geq 1.$ A sufficient condition for (\ref{eqfiniteness})
to hold is that
\begin{equation}
\sum_{k=1}^{\infty
}\mbox{trace}\left(\mathbf{W}_{k}^{T}\boldsymbol{\Lambda}_{k}^{-1}\mathbf{W}_{k}\right)<
\infty,\label{eqsc0} \end{equation} \noindent since from equation
(\ref{tracecvo}), $\sum_{k=1}^{\infty
}\mbox{trace}\left(\boldsymbol{\Lambda }_{k}\right)<\infty.$ In
addition, we restrict out attention to the matrix operators
satisfying (\ref{weightop})--(\ref{eqfiniteness}) and admitting an
inverse matrix operator  $\mathbf{W}^{-1}:H^{n}\longrightarrow H^{n}$
with  $\mathbf{W}\mathbf{W}^{-1}=\mathbf{I}_{H^{n}}.$

One  can easily check  that the  total  sum of squares
$\widetilde{\textbf{SST}}$ for model (\ref{tdm}) is almost surely finite under
(\ref{eqfiniteness}), i.e.,
\begin{eqnarray}
E[\widetilde{\mbox{\textbf{SST}}}] &=&\sum_{k=1}^{\infty }
\mbox{trace}\left(
\mathbf{W}_{k}^{T}\boldsymbol{\Lambda}_{k}^{-1}\mathbf{W}_{k}\boldsymbol{\Lambda}_{k}\right)+\boldsymbol{\beta
}^{T}_{k}\mathbf{X}^{T}\mathbf{W}_{k}^{T}\boldsymbol{\Lambda}_{k}^{-1}\mathbf{W}_{k}\mathbf{X}\boldsymbol{\beta
}_{k} <\infty. \nonumber
\end{eqnarray}
For illustration purposes, we briefly describe a simple way for the
construction of weight  matrix operator $\mathbf{W}$ satisfying  (\ref{eqsc0}).

\subsubsection{Construction of $\mathbf{W}$}
\label{constWpesos} Let us start assuming as given in Remark
\ref{kev} that $\{\phi_{k},\ k\geq 1\}$ and
$\{\boldsymbol{\Lambda}_{k},\ k\geq 1\}$ are known. Then, for each
$k\geq 1,$ we consider the spectral diagonalization of matrix
$\boldsymbol{\Lambda}_{k}$ in terms of its eigenvectors
$\{\psi_{ki},\ i=1,\dots,n\}$ and eigenvalues
$\{\omega_{p}(\boldsymbol{\Lambda}_{k}),\ p=1,\dots,n\}.$ That is,
for each $k\geq 1,$
\begin{equation}\boldsymbol{\Lambda}_{k}=\boldsymbol{\Psi}_{k}\boldsymbol{\Omega}_{k}(\boldsymbol{\Lambda}_{k})\boldsymbol{\Psi}_{k}^{T},
\label{eqseig} \end{equation} \noindent where
$\boldsymbol{\Omega}_{k}(\boldsymbol{\Lambda}_{k})$ is a diagonal
matrix whose non-null elements are
$\omega_{p}(\boldsymbol{\Lambda}_{k}),$ $p=1,\dots,n,$  and where
$\boldsymbol{\Psi}_{k}^{T}\boldsymbol{\Psi}_{k}=\mathbf{I}_{n\times
n},$ with $\boldsymbol{\Psi}_{k}$ having as columns the eigenvectors
of $\boldsymbol{\Lambda}_{k}.$ For each $k\geq 1,$  matrix
$\mathbf{W}_{k}$ can then be constructed in terms of matrix
$\boldsymbol{\Psi}_{k}$ as follows: \begin{equation}\mathbf{W}_{k}=
\boldsymbol{\Psi}_{k}\boldsymbol{\Omega}_{k}(\mathbf{W}_{k})\boldsymbol{\Psi}_{k}^{T},\label{cw}\end{equation}
\noindent where $\boldsymbol{\Omega}_{k}(\mathbf{W}_{k})$ is a
diagonal matrix whose non-null elements
$\omega_{p}(\mathbf{W}_{k}),$ $p=1,\dots,n,$  satisfy  the following
condition: As $k\rightarrow \infty,$
\begin{equation}\omega_{p}(\mathbf{W}_{k})=\mathcal{O}\left(k^{-\frac{\widetilde{\rho
}(p)+\varrho (p)}{2}}\right), \label{c1asSST} \end{equation}
\noindent with $\varrho(p)>1,$ for every
$p=1,\dots,n,$ and
$\omega_{p}(\boldsymbol{\Lambda}_{k})\geq C(k,p), \quad \mbox{with}\quad C(k,p)=\mathcal{O}\left(k^{-\widetilde{\rho }
(p)}\right),\quad  k\rightarrow
\infty,\ p=1,\dots,n.$
 A direct example of matrix $\boldsymbol{\Omega}_{k}(\mathbf{W}_{k})$ is  obtained from condition (\ref{c1asSST}) by considering
$\omega_{p}(\mathbf{W}_{k})=k^{-\frac{\widetilde{\rho}(p)
+\varrho }{2}},
$ with $\varrho >1,$ for $p=1,\dots,n,$ and  for each $k\geq 1.$

We can also simplify the conditions required  in the construction of $\mathbf{W}$ by the choice
of parameter $\widetilde{\rho
}(p)$ independently of $p,$ with $\varrho >1.$ Since
from (\ref{tracecvo}),
$\sum_{k=1}^{\infty
}\sum_{p=1}^{n}\omega_{p}\left(\boldsymbol{\Lambda}_{k}\right)<\infty,$
 let, for each $k\geq 1,$
$\mu_{k}(\boldsymbol{\Lambda}_{k})=\min_{p=1,\dots,n} \omega_{p}\left(\boldsymbol{\Lambda}_{k}\right),$ where $\mu_{k}(\boldsymbol{\Lambda}_{k})\geq
\widetilde{C}(k),\quad \mbox{with}\quad \widetilde{C}(k)=\mathcal{O}(k^{-\rho }),$
\noindent for certain $\rho >1,$ as $k\rightarrow
\infty,$ one can consider \begin{equation}\max_{p=1,\dots,n}
\omega_{p}\left(\mathbf{W}_{k}\right)\leq M k^{-\frac{\rho +\varrho
}{2} },\quad \rho >1,\ \varrho >1,\label{c2asSST} \end{equation}
\noindent for certain $M>0.$ Note that parameter $\rho $ also
characterizes the order of divergence of the sequence
$\widetilde{\mu}_{k}(\boldsymbol{\Lambda}_{k}^{-1})=\max_{p=1,\dots,n}\omega_{p}\left(\boldsymbol{\Lambda}_{k}^{-1}\right),\quad k\geq 1,$
\noindent since, for each $k\geq 1,$  the eigenvalues
$\omega_{p}\left(\boldsymbol{\Lambda}_{k}^{-1}\right),$
$p=1,\dots,n,$ of $\boldsymbol{\Lambda}_{k}^{-1}$ are given by
$\omega_{p}\left(\boldsymbol{\Lambda}_{k}^{-1}\right)=\frac{1}{\omega_{p}(\boldsymbol{\Lambda}_{k})},\quad p=1,\dots,n,\quad k\geq
1.$
\subsection{Almost surely finiteness of the
functional components of variance} \label{sec42}
\begin{proposition}
\label{SST} Under  {\bfseries Assumption A0},  and conditions
(\ref{dmsc}) and (\ref{c1asSST}) (respectively,
 (\ref{c2asSST})), for $\mathbf{W}$ satisfying
$\mathbf{W}_{k}=\mathbf{W}_{k}^{T},$ we  have
$E[\widetilde{\mbox{\textbf{SST}}}]<\infty .$ Consequently,
$\widetilde{\mbox{\textbf{SST}}}$ is almost surely finite.
\end{proposition}
The proof of this result is given in Appendix A.

The sum of squares due to regression
($\widetilde{\mbox{\textbf{SSR}}}$)  for the transformed data model
is given  by
\begin{eqnarray}\widetilde{\mbox{\textbf{SSR}}}&=&\widetilde{\mbox{\textbf{SST}}}-\widetilde{\mbox{\textbf{SSE}}}=
R_{\boldsymbol{\varepsilon}\boldsymbol{\varepsilon}}^{-1}(\widetilde{\mathbf{Y}})(\widetilde{\mathbf{Y}})
-
R_{\boldsymbol{\varepsilon}\boldsymbol{\varepsilon}}^{-1}\left(\widetilde{\mathbf{Y}}-\mathbf{W}\mathbf{X}\widehat{\boldsymbol{\beta
}}\right)\left(\widetilde{\mathbf{Y}}-\mathbf{W}\mathbf{X}\widehat{\boldsymbol{\beta
}}\right)\nonumber\\
&=&\sum_{k=1}^{\infty
}\mathbf{Y}_{k}^{T}\mathbf{W}_{k}^{T}\boldsymbol{\Lambda}_{k}^{-1}\mathbf{W}_{k}\mathbf{Y}_{k}-\mathbf{Y}_{k}^{T}\mathbf{W}_{k}^{T}\boldsymbol{\mathcal{M}}_{k}^{T}\boldsymbol{\Lambda
}_{k}^{-1}\boldsymbol{\mathcal{M}}_{k}\mathbf{W}_{k}\mathbf{Y}_{k}\nonumber\\
&=&\sum_{k=1}^{\infty }
\mathbf{Y}_{k}^{T}[\mathbf{W}_{k}^{T}\boldsymbol{\Lambda}_{k}^{-1}\mathbf{W}_{k}-\mathbf{W}_{k}^{T}\boldsymbol{\mathcal{M}}_{k}^{T}\boldsymbol{\Lambda}_{k}^{-1}
\boldsymbol{\mathcal{M}}_{k}\mathbf{W}_{k}]\mathbf{Y}_{k}.
 \label{ssr22}
\end{eqnarray}
For the almost surely finiteness of the expected sum of squares due
to regression  it is sufficient to consider
\begin{equation}\omega_{p}(\mathbf{W}_{k})=\mathcal{O}\left(k^{-(\widetilde{\rho
}(p)+\varrho (p))}\right), \quad k\rightarrow \infty,
\label{c1asSST2}
\end{equation}
\noindent in equation (\ref{c1asSST}) (respectively to consider
$\max_{p=1,\dots,n} \omega_{p}\left(\mathbf{W}_{k}\right)\leq M
k^{-\rho +\varrho  },$  for $\rho >1$  and $\varrho
>1,$ in (\ref{c2asSST})). In particular an example of matrix operator
$\boldsymbol{W}$ can be constructed from the identity
$\omega_{p}(\mathbf{W}_{k})= k^{-\widetilde{\rho }(p)+\varrho (p)},$
for $p=1,\dots,n,$ and $k\geq 1,$ where $\widetilde{\rho }(p)$ and
$\varrho(p)$ are given as in Section  \ref{constWpesos}. This
construction of $\mathbf{W}$ ensures that
$\sum_{k=1}^{\infty
}\mathrm{trace}\left(\boldsymbol{\Lambda}_{k}^{-1}\mathbf{W}_{k}\right)<\infty,$
\noindent  leading,  under suitable additional  conditions, to
$E[\widetilde{\mbox{\textbf{SSR}}}]<\infty ,$ as given in the
following proposition.
\begin{proposition}
\label{SSR} Under  {\bfseries Assumption A0},  and conditions
(\ref{dmsc}) and  (\ref{c1asSST2}), for $\mathbf{W}$ satisfying
$\mathbf{W}_{k}=\mathbf{W}_{k}^{T},$ we  have
$E[\widetilde{\mbox{\textbf{SSR}}}]<\infty .$ Consequently,
$\widetilde{\mbox{\textbf{SSR}}}$ is almost surely finite.
\end{proposition}
\noindent  The proof of this result is given in the Appendix B.

Finally, the almost surely finiteness of
$\widetilde{\mbox{\textbf{SSE}}}$ follows from Propositions
\ref{SST} and \ref{SSR}, as proven in Appendix C.

\section{Infinite-dimensional distribution of the functional
components of variance} \label{sec5} This section provides the
moment generating and characteristic functio-\linebreak nals of the
statistics $\widetilde{\mbox{\textbf{SST}}},$
$\widetilde{\mbox{\textbf{SSR}}}$ and
$\widetilde{\mbox{\textbf{SSE}}}.$

\subsection{Moment generating functions of the  variance components}
The following result establishes sufficient conditions for the
existence of the moment generating functionals of the statistics
$\widetilde{\mbox{\textbf{SST}}},$ $\widetilde{\mbox{\textbf{SSR}}}$
and $\widetilde{\mbox{\textbf{SSE}}},$ in the transformed functional
data model.
\begin{theorem}
\label{th1} Let us consider that {\bfseries Assumption A0}, and
equation (\ref{dmsc}) are satisfied. Assume also that $\mathbf{W},$
constructed in (\ref{cw}), is strictly positive definite, and that
equation  (\ref{c1asSST2}) to hold. Furthermore,   for each $k\geq
1,$ the elements of the  eigenvalues systems
\begin{eqnarray}& & \{\xi_{i}\left(
\mathbf{W}_{k}^{T}\boldsymbol{\Lambda}_{k}^{-1}\mathbf{W}_{k}
\boldsymbol{\Lambda}_{k}\right),\ i=1,\dots,n\},\nonumber\\
& &\{\xi_{i}\left(\mathbf{W}_{k}^{T}\boldsymbol{\Lambda
}_{k}^{-1}\mathbf{X}(\mathbf{X}^{T}\boldsymbol{\Lambda
}_{k}^{-1}\mathbf{X})^{-1}\mathbf{X}^{T}\boldsymbol{\Lambda
}_{k}^{-1}\mathbf{W}_{k}\boldsymbol{\Lambda }_{k}\right),\
i=1,\dots,n\},\nonumber\\
& &\{\xi_{i}\left(\left(\mathbf{W}_{k}^{T}\boldsymbol{\Lambda
}_{k}^{-1}\mathbf{W}_{k}-\mathbf{W}_{k}^{T}\boldsymbol{\Lambda
}_{k}^{-1}\mathbf{X}(\mathbf{X}^{T}\boldsymbol{\Lambda
}_{k}^{-1}\mathbf{X})^{-1}\mathbf{X}^{T}\boldsymbol{\Lambda
}_{k}^{-1}\mathbf{W}_{k}\right)\boldsymbol{\Lambda }_{k}\right),\
i=1,\dots,n\}\nonumber \end{eqnarray} \noindent of matrices
$\mathbf{W}_{k}^{T}\boldsymbol{\Lambda}_{k}^{-1}\mathbf{W}_{k}
\boldsymbol{\Lambda}_{k},$ $\mathbf{W}_{k}^{T}\boldsymbol{\Lambda
}_{k}^{-1}\mathbf{X}(\mathbf{X}^{T}\boldsymbol{\Lambda
}_{k}^{-1}\mathbf{X})^{-1}\mathbf{X}^{T}\boldsymbol{\Lambda
}_{k}^{-1}\mathbf{W}_{k}\boldsymbol{\Lambda }_{k},$ and \linebreak
$\left(\mathbf{W}_{k}^{T}\boldsymbol{\Lambda
}_{k}^{-1}\mathbf{W}_{k}-\mathbf{W}_{k}^{T}\boldsymbol{\Lambda
}_{k}^{-1}\mathbf{X}(\mathbf{X}^{T}\boldsymbol{\Lambda
}_{k}^{-1}\mathbf{X})^{-1}\mathbf{X}^{T}\boldsymbol{\Lambda
}_{k}^{-1}\mathbf{W}_{k}\right)\boldsymbol{\Lambda }_{k},$
respectively,
 are considered to be strictly less than one. Then,  the moment
generating functions of $\widetilde{\mbox{\textbf{SST}}},$
$\widetilde{\mbox{\textbf{SSR}}}$ and
$\widetilde{\mbox{\textbf{SSE}}}$ are given by
\begin{eqnarray}
&& M_{\widetilde{\mbox{\textbf{SST}}}}(t/2)=E[\exp(t/2(\widetilde{\mbox{\textbf{SST}}}))]\nonumber\\
&&=\prod_{k=1}^{\infty}\left[\mbox{det}\left(\mathbf{I}_{n\times
n}-t\mathbf{W}_{k}^{T}\boldsymbol{\Lambda}_{k}^{-1}\mathbf{W}_{k}\boldsymbol{\Lambda}_{k}\right)\right]^{-1/2}\nonumber\\
&& \times \exp\left(-\frac{1}{2}\boldsymbol{\beta
}^{T}_{k}\mathbf{X}^{T}\left(\mathbf{I}_{n\times
n}-(\mathbf{I}_{n\times
n}-t\mathbf{W}_{k}^{T}\boldsymbol{\Lambda}_{k}^{-1}\mathbf{W}_{k}\boldsymbol{\Lambda}_{k})^{-1}\right)
\boldsymbol{\Lambda}_{k}^{-1}\mathbf{X}\boldsymbol{\beta
}_{k}\right)\nonumber\\ \label{mgfSSTbb}
\end{eqnarray}
\begin{eqnarray}
&& M_{\widetilde{\mbox{\textbf{SSR}}}}(t/2)=E[\exp(t/2(\widetilde{\mbox{\textbf{SSR}}}))]\nonumber\\
&&=\prod_{k=1}^{\infty}\left[\mbox{det}\left(\mathbf{I}_{n\times
n}-t\mathbf{W}_{k}^{T}\boldsymbol{\Lambda
}_{k}^{-1}\mathbf{X}(\mathbf{X}^{T}\boldsymbol{\Lambda
}_{k}^{-1}\mathbf{X})^{-1}\mathbf{X}^{T}\boldsymbol{\Lambda
}_{k}^{-1}\mathbf{W}_{k}\boldsymbol{\Lambda }_{k}\right)\right]^{-1/2}\nonumber\\
&&\times \exp\left(-\frac{1}{2}\boldsymbol{\beta
}^{T}_{k}\mathbf{X}^{T}\left(\mathbf{I}_{n\times
n}-(\mathbf{I}_{n\times n}\right.\right.\nonumber\\
&& \hspace*{2cm}\left.\left.-t\mathbf{W}_{k}^{T}\boldsymbol{\Lambda
}_{k}^{-1}\mathbf{X}(\mathbf{X}^{T}\boldsymbol{\Lambda
}_{k}^{-1}\mathbf{X})^{-1}\mathbf{X}^{T}\boldsymbol{\Lambda
}_{k}^{-1}\mathbf{W}_{k}\boldsymbol{\Lambda }_{k})^{-1}\right)
\boldsymbol{\Lambda}_{k}^{-1}\mathbf{X}\boldsymbol{\beta
}_{k}\right)\nonumber\\
\label{mgfSSRbb}
\end{eqnarray}
\begin{eqnarray}
&& M_{\widetilde{\mbox{\textbf{SSE}}}}(t/2)=E[\exp(t/2(\widetilde{\mbox{\textbf{SSE}}}))]\nonumber\\
&&=\prod_{k=1}^{\infty}\left[\mbox{det}\left(\mathbf{I}_{n\times
n}-t\left(\mathbf{W}_{k}^{T}\boldsymbol{\Lambda
}_{k}^{-1}\mathbf{W}_{k}-\mathbf{W}_{k}^{T}\boldsymbol{\Lambda
}_{k}^{-1}\mathbf{X}(\mathbf{X}^{T}\boldsymbol{\Lambda
}_{k}^{-1}\mathbf{X})^{-1}\mathbf{X}^{T}\boldsymbol{\Lambda
}_{k}^{-1}\mathbf{W}_{k}\right)\boldsymbol{\Lambda }_{k}\right)\right]^{-1/2}\nonumber\\
&&\times \exp\left(-\frac{1}{2}\boldsymbol{\beta
}^{T}_{k}\mathbf{X}^{T}\left(\mathbf{I}_{n\times
n}-(\mathbf{I}_{n\times n}\right.\right.\nonumber\\
&& \left.\left.-t\left(\mathbf{W}_{k}^{T}\boldsymbol{\Lambda
}_{k}^{-1}\mathbf{W}_{k}-\mathbf{W}_{k}^{T}\boldsymbol{\Lambda
}_{k}^{-1}\mathbf{X}(\mathbf{X}^{T}\boldsymbol{\Lambda
}_{k}^{-1}\mathbf{X})^{-1}\mathbf{X}^{T}\boldsymbol{\Lambda
}_{k}^{-1}\mathbf{W}_{k}\right)\boldsymbol{\Lambda
}_{k})^{-1}\right)
\boldsymbol{\Lambda}_{k}^{-1}\mathbf{X}\boldsymbol{\beta
}_{k}\right).\nonumber\\
\label{mgfSSEbb}
\end{eqnarray}
\end{theorem}
\begin{proof}
 We will apply that for a  $n\times 1$ Gaussian vector
 $\mathbf{y}\sim \mathcal{N}(\boldsymbol{\mu},
\boldsymbol{\Sigma}),$ the moment generating function of
$\mathbf{y}^{T}\mathbf{A}\mathbf{y}$ admits the following
expression (see, for example, \cite{Hocking},
 pp. 600-608):
\begin{eqnarray}
& & E[\exp(t(\mathbf{y}^{T}\mathbf{A}\mathbf{y}))]
=\left[\mbox{det}\left(\mathbf{I}_{n\times
n}-2t\mathbf{A}\boldsymbol{\Sigma}\right)\right]^{-1/2}\nonumber\\
& &\hspace*{1cm}\times
\exp\left(-\frac{1}{2}\boldsymbol{\mu}^{T}(\mathbf{I}_{n\times
n}-(\mathbf{I}_{n\times
n}-2t\mathbf{A}\boldsymbol{\Sigma})^{-1})\boldsymbol{\Sigma}^{-1}\boldsymbol{\mu}\right).
\label{eqmgf1ttt}
\end{eqnarray}
\noindent In addition, under {\bfseries Assumption A0}, the elements of the
sequences
\begin{equation}\mathbf{Y}^{T}_{k}\mathbf{W}_{k}^{T}\boldsymbol{\Lambda}_{k}^{-1}\mathbf{W}_{k}\mathbf{Y}_{k},\quad k\geq 1,\label{es1}
\end{equation}
 \noindent  and  \begin{equation}\mathbf{Y}^{T}_{k}\mathbf{W}_{k}^{T}
 \boldsymbol{\mathcal{M}}_{k}^{T}\boldsymbol{\Lambda}_{k}^{-1}\boldsymbol{\mathcal{M}}_{k}\mathbf{W}_{k}\mathbf{Y}_{k},\quad k\geq 1\label{es2}
\end{equation}
 \noindent  are mutually independent. Then, for each $k\geq 1,$ applying (\ref{eqmgf1ttt}) to $n\times 1$ Gaussian vector $\mathbf{Y}_{k}\sim
 \mathcal{N}([\mathbf{X}\boldsymbol{\beta }]_{k},\boldsymbol{\Lambda}_{k}),$ and to  matrices
\begin{eqnarray}
\mathbf{A}_{\widetilde{\mbox{\textbf{SST}}}}^{k} &=&
 \mathbf{W}_{k}^{T}\boldsymbol{\Lambda}_{k}^{-1}\mathbf{W}_{k}\nonumber\\
 \mathbf{A}_{\widetilde{\mbox{\textbf{SSR}}}}^{k}&=&\mathbf{W}_{k}^{T}\boldsymbol{\Lambda
}_{k}^{-1}\mathbf{X}(\mathbf{X}^{T}\boldsymbol{\Lambda
}_{k}^{-1}\mathbf{X})^{-1}\mathbf{X}^{T}\boldsymbol{\Lambda
}_{k}^{-1}\mathbf{W}_{k}\nonumber\\
\mathbf{A}_{\widetilde{\mbox{\textbf{SSE}}}}^{k}&=&\left(\mathbf{W}_{k}^{T}\boldsymbol{\Lambda
}_{k}^{-1}\mathbf{W}_{k}-\mathbf{W}_{k}^{T}\boldsymbol{\Lambda
}_{k}^{-1}\mathbf{X}(\mathbf{X}^{T}\boldsymbol{\Lambda
}_{k}^{-1}\mathbf{X})^{-1}\mathbf{X}^{T}\boldsymbol{\Lambda
}_{k}^{-1}\mathbf{W}_{k}\right),
 \label{idmgf}
 \end{eqnarray}
 \noindent playing the role of  matrix $\mathbf{A},$ for each element of the infinite series defining $\widetilde{\mbox{\textbf{SST}}},$ $\widetilde{\mbox{\textbf{SSR}}},$
 and $\widetilde{\mbox{\textbf{SSE}}},$ respectively, we obtain equations (\ref{mgfSSTbb}), (\ref{mgfSSRbb}) and (\ref{mgfSSEbb}) from  the independence of the elements of the sequences (\ref{es1})
 and (\ref{es2}).

  In equation  (\ref{mgfSSTbb}), the infinite product
$\prod_{k=1}^{\infty}\left[\mbox{det}\left(\mathbf{I}_{n\times
n}-t\mathbf{W}_{k}^{T}\boldsymbol{\Lambda}_{k}^{-1}\mathbf{W}_{k}\boldsymbol{\Lambda}_{k}\right)\right]^{-1/2}$
is finite since it provides the negative square root of
 the Fredholm determinant of operator
$\mathbf{W}^{*}\boldsymbol{R}_{\boldsymbol{\varepsilon}\boldsymbol{\varepsilon}}^{-1}\mathbf{W}\boldsymbol{R}_{\boldsymbol{\varepsilon}\boldsymbol{\varepsilon}}$
at point $t.$
 From   condition (\ref{c1asSST2}), operator \linebreak
$\mathbf{W}^{*}\boldsymbol{R}_{\boldsymbol{\varepsilon}\boldsymbol{\varepsilon}}^{-1}\mathbf{W}\boldsymbol{R}_{\boldsymbol{\varepsilon}\boldsymbol{\varepsilon}}$
is in the trace class (see also  equation (\ref{c1asSST}) and  Appendix A). Hence, its Fredholm
determinant is finite for
\begin{equation}t<\frac{1}{\mbox{trace}\left(\mathbf{W}^{*}\boldsymbol{R}_{\boldsymbol{\varepsilon}\boldsymbol{\varepsilon}}^{-1}\mathbf{W}\boldsymbol{R}_{\boldsymbol{\varepsilon}\boldsymbol{\varepsilon}}\right)}
\label{eqSSTttt}
\end{equation}
 \noindent (see, for example, \cite{SimonB}, Chapter 5, pp. 47-48, equation
(5.12)).

In a similar way, it can be proved in  equation (\ref{mgfSSRbb})
that the infinite product
$\prod_{k=1}^{\infty}\left[\mbox{det}\left(\mathbf{I}_{n\times
n}-t\mathbf{W}_{k}^{T}\boldsymbol{\Lambda
}_{k}^{-1}\mathbf{X}(\mathbf{X}^{T}\boldsymbol{\Lambda
}_{k}^{-1}\mathbf{X})^{-1}\mathbf{X}^{T}\boldsymbol{\Lambda
}_{k}^{-1}\mathbf{W}_{k}\boldsymbol{\Lambda
}_{k}\right)\right]^{-1/2}$ \noindent is finite, since it provides
the negative square root of the Fredholm determinant of operator
$\mathbf{W}^{*}\boldsymbol{R}_{\boldsymbol{\varepsilon}\boldsymbol{\varepsilon}}^{-1}\mathbf{X}(\mathbf{X}^{T}\boldsymbol{R}_{\boldsymbol{\varepsilon}\boldsymbol{\varepsilon}}^{-1}
\mathbf{X})^{-1}\mathbf{X}^{T}\boldsymbol{R}_{\boldsymbol{\varepsilon}\boldsymbol{\varepsilon}}^{-1}\mathbf{W}\boldsymbol{R}_{\boldsymbol{\varepsilon}\boldsymbol{\varepsilon}},$
at point $t.$
 Note that, again,  from (\ref{c1asSST2}), under
(\ref{dmsc}), operator
$$\mathbf{W}^{*}\boldsymbol{R}_{\boldsymbol{\varepsilon}\boldsymbol{\varepsilon}}^{-1}\mathbf{X}(\mathbf{X}^{T}\boldsymbol{R}_{\boldsymbol{\varepsilon}\boldsymbol{\varepsilon}}^{-1}
\mathbf{X})^{-1}\mathbf{X}^{T}\boldsymbol{R}_{\boldsymbol{\varepsilon}\boldsymbol{\varepsilon}}^{-1}\mathbf{W}\boldsymbol{R}_{\boldsymbol{\varepsilon}\boldsymbol{\varepsilon}}$$
\noindent is in the trace class (see Appendix B), and hence, its
Fredholm determinant is finite for
\begin{equation}t<\frac{1}{\mbox{trace}\left(\mathbf{W}^{*}\boldsymbol{R}_{\boldsymbol{\varepsilon}\boldsymbol{\varepsilon}}^{-1}\mathbf{X}(\mathbf{X}^{T}\boldsymbol{R}_{\boldsymbol{\varepsilon}\boldsymbol{\varepsilon}}^{-1}
\mathbf{X})^{-1}\mathbf{X}^{T}\boldsymbol{R}_{\boldsymbol{\varepsilon}\boldsymbol{\varepsilon}}^{-1}
\mathbf{W}\boldsymbol{R}_{\boldsymbol{\varepsilon}\boldsymbol{\varepsilon}}\right)}\label{fcsst}
\end{equation}
 (see, for example, \cite{SimonB}, Chapter 5, pp. 47-48, equation
(5.12)).

Finally, in  equation (\ref{mgfSSEbb}), the negative square root of
the Fredholm determinant at point $t$ of the trace operator
$$\left(\mathbf{W}^{*}\boldsymbol{R}_{\boldsymbol{\varepsilon}\boldsymbol{\varepsilon}}^{-1}\mathbf{W}-\mathbf{W}^{*}
\boldsymbol{R}_{\boldsymbol{\varepsilon}\boldsymbol{\varepsilon}}^{-1}\mathbf{X}(\mathbf{X}^{T}\boldsymbol{R}_{\boldsymbol{\varepsilon}\boldsymbol{\varepsilon}}^{-1}\mathbf{X})^{-1}\mathbf{X}^{T}
\boldsymbol{R}_{\boldsymbol{\varepsilon}\boldsymbol{\varepsilon}}^{-1}\mathbf{W}\right)\boldsymbol{R}_{\boldsymbol{\varepsilon}\boldsymbol{\varepsilon}},$$
is given by
$$\prod_{k=1}^{\infty}\left[\mbox{det}\left(\mathbf{I}_{n\times
n}-t\left(\mathbf{W}_{k}^{T}\boldsymbol{\Lambda
}_{k}^{-1}\mathbf{W}_{k}-\mathbf{W}_{k}^{T}\boldsymbol{\Lambda
}_{k}^{-1}\mathbf{X}(\mathbf{X}^{T}\boldsymbol{\Lambda
}_{k}^{-1}\mathbf{X})^{-1}\mathbf{X}^{T}\boldsymbol{\Lambda
}_{k}^{-1}\mathbf{W}_{k}\right)\boldsymbol{\Lambda
}_{k}\right)\right]^{-1/2},$$  \noindent which is finite for
\begin{equation}t<\frac{1}{\mbox{trace}\left(\left(\mathbf{W}^{*}\boldsymbol{R}_{\boldsymbol{\varepsilon}\boldsymbol{\varepsilon}}^{-1}\mathbf{W}-\mathbf{W}^{*}
\boldsymbol{R}_{\boldsymbol{\varepsilon}\boldsymbol{\varepsilon}}^{-1}\mathbf{X}(\mathbf{X}^{T}\boldsymbol{R}_{\boldsymbol{\varepsilon}\boldsymbol{\varepsilon}}^{-1}\mathbf{X})^{-1}\mathbf{X}^{T}
\boldsymbol{R}_{\boldsymbol{\varepsilon}\boldsymbol{\varepsilon}}^{-1}\mathbf{W}\right)
\boldsymbol{R}_{\boldsymbol{\varepsilon}\boldsymbol{\varepsilon}}\right)}\label{tSSEtracenorm}
\end{equation}
\noindent (see, for example, \cite{SimonB}, Chapter 5, pp. 47-48,
equation (5.12)).

We now study the  finiteness of the second factor at the right-hand
side of equations (\ref{mgfSSTbb}), (\ref{mgfSSRbb}) and
(\ref{mgfSSEbb}), given in terms  of negative exponential functions.
Specifically,   in equation (\ref{mgfSSTbb}),
consider
\begin{eqnarray}t&<&
K_{\widetilde{\mbox{\textbf{SST}}}}=\frac{1}{\max_{k\geq 1;\
i=1,\dots,n}\xi_{i}\left(
\mathbf{W}_{k}^{T}\boldsymbol{\Lambda}_{k}^{-1}\mathbf{W}_{k}
\boldsymbol{\Lambda}_{k}\right)} \nonumber\\&& \hspace*{1cm}\times
\left[1-\frac{1}{1-\max_{k\geq 1; \ i=1,\dots,n}\xi_{i}\left(
\mathbf{W}_{k}^{T}\boldsymbol{\Lambda}_{k}^{-1}\mathbf{W}_{k}
\boldsymbol{\Lambda}_{k}\right)}\right]. \label{egfcsst}
\end{eqnarray}
\noindent  where, as before, for each $k\geq 1,$
$\xi_{i}(\mathbf{A}_{k})$ denotes the $i$th eigenvalue of $n\times
n$ matrix $\mathbf{A}_{k},$ appearing in the series representation
of a matrix operator $\mathcal{\boldsymbol{A}}$ defined on
$\mathcal{H}=H^{n},$ such that
$\boldsymbol{\Phi}^{*}\mathcal{\boldsymbol{A}}\boldsymbol{\Phi}=\left(\mathbf{A}_{k}\right)_{k\geq
1}.$ Since $$K_{\widetilde{\mbox{\textbf{SST}}}}
\leq \frac{1}{\xi_{i}\left(
\mathbf{W}_{k}^{T}\boldsymbol{\Lambda}_{k}^{-1}\mathbf{W}_{k}
\boldsymbol{\Lambda}_{k}\right)} \left[1-\frac{1}{1-\xi_{i}\left(
\mathbf{W}_{k}^{T}\boldsymbol{\Lambda}_{k}^{-1}\mathbf{W}_{k}
\boldsymbol{\Lambda}_{k}\right)}\right],$$ \noindent for every
$i=1,\dots,n,$ and $k\geq 1,$  for $t<
K_{\widetilde{\mbox{\textbf{SST}}}},$ we obtain
\begin{eqnarray}
&& \exp\left(-\frac{1}{2}\sum_{k=1}^{\infty }\boldsymbol{\beta
}^{T}_{k}\mathbf{X}^{T}\left(\mathbf{I}_{n\times
n}-(\mathbf{I}_{n\times
n}-t\mathbf{W}_{k}^{T}\boldsymbol{\Lambda}_{k}^{-1}\mathbf{W}_{k}\boldsymbol{\Lambda}_{k})^{-1}\right)
\boldsymbol{\Lambda}_{k}^{-1}\mathbf{X}\boldsymbol{\beta
}_{k}\right) \nonumber
\\
& & =\exp\left(-\frac{1}{2}\sum_{k=1}^{\infty
}\sum_{i=1}^{n}[\boldsymbol{\Psi}_{k}^{T}[\mathbf{X}\boldsymbol{\beta
}]_{k}]_{i}^{2}\left[1-\frac{1}{1-t\xi_{i}\left(
\mathbf{W}_{k}^{T}\boldsymbol{\Lambda}_{k}^{-1}\mathbf{W}_{k}
\boldsymbol{\Lambda}_{k}\right)}\right]\xi_{i}\left(\boldsymbol{\Lambda}_{k}^{-1}\right)\right)
\nonumber\\
& &  \leq \exp\left(-\frac{1}{2}\sum_{k=1}^{\infty
}\sum_{i=1}^{n}[\boldsymbol{\Psi}_{k}^{T}[\mathbf{X}\boldsymbol{\beta
}]_{k}]_{i}^{2}\left[1- \right.\right.\nonumber\\
& & \left.\left.  \frac{1}{1- \frac{1}{\xi_{i}\left(
\mathbf{W}_{k}^{T}\boldsymbol{\Lambda}_{k}^{-1}\mathbf{W}_{k}
\boldsymbol{\Lambda}_{k}\right)} \left[1-\frac{1}{1-\xi_{i}\left(
\mathbf{W}_{k}^{T}\boldsymbol{\Lambda}_{k}^{-1}\mathbf{W}_{k}
\boldsymbol{\Lambda}_{k}\right)}\right] \xi_{i}\left(
\mathbf{W}_{k}^{T}\boldsymbol{\Lambda}_{k}^{-1}\mathbf{W}_{k}
\boldsymbol{\Lambda}_{k}\right)}\right]\xi_{i}\left(\boldsymbol{\Lambda}_{k}^{-1}\right)\right)
\nonumber
\end{eqnarray}
\begin{eqnarray}
& & =
 \exp\left(-\frac{1}{2} \sum_{k=1}^{\infty
}\boldsymbol{\beta }^{T}_{k}\mathbf{X}^{T}
\mathbf{W}_{k}^{T}\boldsymbol{\Lambda}_{k}^{-1}\mathbf{W}_{k}\mathbf{X}\boldsymbol{\beta
}_{k}\right), \nonumber\\
\label{eqmgf1}
\end{eqnarray} \noindent where, for each $k\geq 1,$
$\boldsymbol{\Psi}_{k}$ is the projection operator into the
eigenvectors of $\boldsymbol{\Lambda}_{k},$ appearing  in equation
(\ref{eqseig}). Also, for each $k\geq 1,$ and for $i=1,\dots,n,$
$[\boldsymbol{\Psi}_{k}^{T}[\mathbf{X}\boldsymbol{\beta }]_{k}]_{i}$
denotes the $i$th projection  of the $n\times 1$ vector
$[\mathbf{X}\boldsymbol{\beta }]_{k}$ with respect to the $i$th
eigenvector  of $\boldsymbol{\Lambda}_{k}.$ Note that we have
applied that, for each $k\geq 1,$ $\mathbf{W}_{k}$ has been
constructed from the same eigenvector system as
$\boldsymbol{\Lambda}_{k}$ (see equations (\ref{eqseig}) and
(\ref{cw})). In particular, $\xi_{i}\left(
\mathbf{W}_{k}^{T}\boldsymbol{\Lambda}_{k}^{-1}\mathbf{W}_{k}
\boldsymbol{\Lambda}_{k}\right)=\xi_{i}\left(
\mathbf{W}_{k}^{T}\boldsymbol{\Lambda}_{k}^{-1}\mathbf{W}_{k}\right)\xi_{i}\left(\boldsymbol{\Lambda}_{k}\right).$
Finally, equation (\ref{eqmgf1}) is finite from
condition (\ref{c1asSST2}), which implies
$$\sum_{k=1}^{\infty
}\mbox{trace}\left(\mathbf{W}_{k}^{T}\boldsymbol{\Lambda}_{k}^{-1}\mathbf{W}_{k}\right)<\infty.$$

  From equations (\ref{eqSSTttt}) and (\ref{egfcsst}), denoting  \begin{equation}\mathrm{IT}_{\widetilde{\mbox{\textbf{SST}}}}= \frac{1}{\mbox{trace}\left(\mathbf{W}^{*}
  \boldsymbol{R}_{\boldsymbol{\varepsilon}\boldsymbol{\varepsilon}}^{-1}\mathbf{W}
\boldsymbol{R}_{\boldsymbol{\varepsilon}\boldsymbol{\varepsilon}}\right)},\label{minsst}
\end{equation}
we have that
 $M_{\widetilde{\mbox{\textbf{SST}}}}(t)$ is
finite for every
$t<\min\{K_{\widetilde{\mbox{\textbf{SST}}}},\mathrm{IT}_{\widetilde{\mbox{\textbf{SST}}}}\}.$
\noindent An analytic continuation argument (see \cite{Lukacs}, Th.
7.1.1) guarantees that $M_{\widetilde{\mbox{\textbf{SST}}}}(t)$
defines the unique limit moment generating function for all real
values of $t.$

Similar arguments to equation (\ref{mgfSSTbb}) can
be applied for the proof of the finiteness of the second negative
exponential factors in equations (\ref{mgfSSRbb}) and
(\ref{mgfSSEbb}), as well as for the existence of the moment
generating functions given in such equations. The details can be
left to the reader, since they can be obtained  straightforward from the above-described steps
 by replacing, for each $k\geq 1,$
  matrix
$\mathbf{W}_{k}^{T}\boldsymbol{\Lambda}_{k}^{-1}\mathbf{W}_{k}
\boldsymbol{\Lambda}_{k}$ by
 matrix
$\mathbf{W}_{k}^{T}\boldsymbol{\Lambda
}_{k}^{-1}\mathbf{X}(\mathbf{X}^{T}\boldsymbol{\Lambda
}_{k}^{-1}\mathbf{X})^{-1}\mathbf{X}^{T}\boldsymbol{\Lambda
}_{k}^{-1}\mathbf{W}_{k}\boldsymbol{\Lambda }_{k},$ in the case of equation (\ref{mgfSSRbb}), and,  in  equation  (\ref{mgfSSEbb}), replacing
it by matrix
$$\left(\mathbf{W}_{k}^{T}\boldsymbol{\Lambda
}_{k}^{-1}\mathbf{W}_{k}-\mathbf{W}_{k}^{T}\boldsymbol{\Lambda
}_{k}^{-1}\mathbf{X}(\mathbf{X}^{T}\boldsymbol{\Lambda
}_{k}^{-1}\mathbf{X})^{-1}\mathbf{X}^{T}\boldsymbol{\Lambda
}_{k}^{-1}\mathbf{W}_{k}\right)\boldsymbol{\Lambda }_{k}.$$
\end{proof}

\subsection{Characteristic functions of the variance components}

 In the derivation of the
results in this section, we apply Proposition 1.2.8 of Chapter 1,
p.14, in \cite{Da Prato}, where the characteristic function of
quadratic forms defined in terms of
 symmetric  operators, and  Hilbert-valued
Gaussian random variables is provided. This result is formulated in Lemma
\ref{lem1} below, for the special case where the Hilbert space
considered is $\mathcal{H}=H^{n}.$
\begin{lemma}
\label{lem1}
 Let  $\mathbf{Y}$ be an
$\mathcal{H}$-valued zero-mean Gaussian random variable  with trace
covariance matrix operator $\mathbf{R}_{\mathbf{Y}\mathbf{Y}}.$ Let
$\mathbf{M}$ be a  symmetric matrix operator  on $\mathcal{H}.$
Assume that $\|\mathbf{R}_{\mathbf{Y}\mathbf{Y}}^{1/2}
\mathbf{M}\mathbf{R}_{\mathbf{Y}\mathbf{Y}}^{1/2}\|_{\mathcal{L}(\mathcal{H})}<1,$
 where $\|\cdot \|_{\mathcal{L}(\mathcal{H})}$ denotes the
norm in the space of bounded linear operators on $\mathcal{H}.$
Then,   for  $\mathbf{b} \in \mathcal{H},$
\begin{eqnarray}& & E\left[\exp\left( \frac{1}{2}\left\langle
\mathbf{M}\mathbf{Y},\mathbf{Y}\right\rangle_{\mathcal{H}}+
\left\langle
\mathbf{b},\mathbf{Y}\right\rangle_{\mathcal{H}}\right)\right] =\left[\mbox{det}\left(\mathbf{I}-
\mathbf{R}_{\mathbf{Y}\mathbf{Y}}^{1/2}
\mathbf{M}\mathbf{R}_{\mathbf{Y}\mathbf{Y}}^{1/2}\right)\right]^{-1/2}
\nonumber\\
& & \times \exp\left\{\frac{1}{2} \left\|\left(\mathbf{I}-
\mathbf{R}_{\mathbf{Y}\mathbf{Y}}^{1/2}
\mathbf{M}\mathbf{R}_{\mathbf{Y}\mathbf{Y}}^{1/2}\right)^{-1/2}
\mathbf{R}_{\mathbf{Y}\mathbf{Y}}^{1/2}
\mathbf{b}\right\|^{2}_{\mathcal{H}}\right\}. \label{eqchfhrv}
\end{eqnarray}
\end{lemma}
 The next result establishes sufficient  conditions for the explicit definition of the characteristic functionals of $\widetilde{\mbox{\textbf{SST}}},$ $\widetilde{\mbox{\textbf{SSR}}}$
and $\widetilde{\mbox{\textbf{SSE}}}.$
\begin{theorem}
\label{characteristicfunctions} Under  {\bfseries Assumption A0},
and conditions   (\ref{dmsc}) and (\ref{c1asSST2}),  the following
assertions hold:
\begin{itemize}
\item[] (i) The characteristic functional of
$\widetilde{\mbox{\textbf{SST}}}$ is defined as
\begin{eqnarray}
&& F_{\widetilde{\mbox{\textbf{SST}}}}(i\omega )=
E\left[\exp\left(i\omega\widetilde{\mbox{\textbf{SST}}}\right)\right]
\nonumber\\
& & =\prod_{k=1}^{\infty}\left[\mbox{det}\left(\mathbf{I}_{n\times
n}-2i\omega
\boldsymbol{\Lambda}_{k}^{1/2}\mathbf{W}_{k}^{T}\boldsymbol{\Lambda}_{k}^{-1}\mathbf{W}_{k}\boldsymbol{\Lambda}_{k}^{1/2}\right)\right]^{-1/2}
\nonumber\\
& &\times \exp\left(-4\omega^{2}\sum_{k=1}^{\infty}\boldsymbol{\beta
}_{k}^{T}\mathbf{X}^{T}\mathbf{W}_{k}^{T}\boldsymbol{\Lambda}_{k}^{-1}\mathbf{W}_{k}\boldsymbol{\Lambda}_{k}^{1/2}\left(\mathbf{I}_{n\times
n}\right.\right.\nonumber\\
& & \hspace*{1cm}\left.\left. -2i\omega
\boldsymbol{\Lambda}_{k}^{1/2}\mathbf{W}_{k}^{T}\boldsymbol{\Lambda}_{k}^{-1}\mathbf{W}_{k}\boldsymbol{\Lambda}_{k}^{1/2}\right)^{-1}
\boldsymbol{\Lambda}_{k}^{1/2}\mathbf{W}_{k}^{T}\boldsymbol{\Lambda}_{k}^{-1}\mathbf{W}_{k}\mathbf{X}\boldsymbol{\beta
}_{k}\right)\nonumber\\
& &  \times \exp\left( i\omega\sum_{k=1}^{\infty}\boldsymbol{\beta
}_{k}^{T}\mathbf{X}^{T}\mathbf{W}_{k}^{T}\boldsymbol{\Lambda}_{k}^{-1}\mathbf{W}_{k}\mathbf{X}\boldsymbol{\beta
}_{k} \right).\nonumber\\
 \label{mgfSST}
\end{eqnarray}

\item[] (ii)
The characteristic functional of $\widetilde{\mbox{\textbf{SSR}}}$
is given by
\begin{eqnarray}
&& F_{\widetilde{\mbox{\textbf{SSR}}}}(i\omega )=
E\left[\exp\left(i\omega
\widetilde{\mbox{\textbf{SSR}}}\right)\right]\nonumber\\
& &  =\prod_{k=1}^{\infty}\left[\mbox{det}\left(\mathbf{I}_{n\times
n}-2i\omega
\boldsymbol{\Lambda}_{k}^{1/2}\mathbf{W}_{k}^{T}\boldsymbol{\Lambda
}_{k}^{-1}\mathbf{X}(\mathbf{X}^{T}\boldsymbol{\Lambda
}_{k}^{-1}\mathbf{X})^{-1}\mathbf{X}^{T}\boldsymbol{\Lambda
}_{k}^{-1}\mathbf{W}_{k}\boldsymbol{\Lambda}_{k}^{1/2}\right)\right]^{-1/2}
\nonumber
\end{eqnarray}
\begin{eqnarray}
& &\times \exp\left(-4\omega^{2}\sum_{k=1}^{\infty}\boldsymbol{\beta
}_{k}^{T}\mathbf{X}^{T}\mathbf{W}_{k}^{T}\boldsymbol{\Lambda
}_{k}^{-1}\mathbf{X}(\mathbf{X}^{T}\boldsymbol{\Lambda
}_{k}^{-1}\mathbf{X})^{-1}\mathbf{X}^{T}\boldsymbol{\Lambda
}_{k}^{-1}\mathbf{W}_{k}\boldsymbol{\Lambda}_{k}^{1/2}\right.\nonumber\\
& & \hspace*{1.5cm}\left.\times \left(\mathbf{I}_{n\times
n}-2i\omega
\boldsymbol{\Lambda}_{k}^{1/2}\mathbf{W}_{k}^{T}\boldsymbol{\Lambda
}_{k}^{-1}\mathbf{X}(\mathbf{X}^{T}\boldsymbol{\Lambda
}_{k}^{-1}\mathbf{X})^{-1}\mathbf{X}^{T}\boldsymbol{\Lambda
}_{k}^{-1}\mathbf{W}_{k}\boldsymbol{\Lambda}_{k}^{1/2}\right)^{-1}
\right.\nonumber\\
& &\hspace*{1.5cm}\left.\times
\boldsymbol{\Lambda}_{k}^{1/2}\mathbf{W}_{k}^{T}\boldsymbol{\Lambda
}_{k}^{-1}\mathbf{X}(\mathbf{X}^{T}\boldsymbol{\Lambda
}_{k}^{-1}\mathbf{X})^{-1}\mathbf{X}^{T}\boldsymbol{\Lambda
}_{k}^{-1}\mathbf{W}_{k}\mathbf{X}\boldsymbol{\beta
}_{k}\right)\nonumber\\
& &\times  \exp\left( i\omega\sum_{k=1}^{\infty}\boldsymbol{\beta
}_{k}^{T}\mathbf{X}^{T}\mathbf{W}_{k}^{T}\boldsymbol{\Lambda
}_{k}^{-1}\mathbf{X}(\mathbf{X}^{T}\boldsymbol{\Lambda
}_{k}^{-1}\mathbf{X})^{-1}\mathbf{X}^{T}\boldsymbol{\Lambda
}_{k}^{-1}\mathbf{W}_{k}\mathbf{X}\boldsymbol{\beta
}_{k} \right).\nonumber\\
\label{mgfSSR}
\end{eqnarray}

\item[] (iii)
The characteristic functional of $\widetilde{\mbox{\textbf{SSE}}}$
can be expressed as
\begin{eqnarray}
&& F_{\widetilde{\mbox{\textbf{SSE}}}}(i\omega)=
E\left[\exp\left(i\omega
\widetilde{\mbox{\textbf{SSE}}}\right)\right]\nonumber\\
& &  =\prod_{k=1}^{\infty}\left[\mbox{det}\left(\mathbf{I}_{n\times
n}-2i\omega
\boldsymbol{\Lambda}_{k}^{1/2}\left(\mathbf{W}^{T}_{k}\boldsymbol{\Lambda}_{k}^{-1}\mathbf{W}_{k}\right.\right.\right.
\nonumber\\
& &\hspace*{2cm}\left.\left.\left.
-\mathbf{W}^{T}_{k}\boldsymbol{\Lambda}_{k}^{-1}
\mathbf{X}(\mathbf{X}^{T}\boldsymbol{\Lambda}_{k}^{-1}\mathbf{X})^{-1}
\mathbf{X}^{T}\boldsymbol{\Lambda}_{k}^{-1}\mathbf{W}_{k}\right)\boldsymbol{\Lambda}_{k}^{1/2}\right)\right]^{-1/2}
\nonumber\\
& &\hspace*{-1cm}\times
\exp\left(-4\omega^{2}\sum_{k=1}^{\infty}\boldsymbol{\beta
}_{k}^{T}\mathbf{X}^{T}\left(\mathbf{W}^{T}_{k}\boldsymbol{\Lambda}_{k}^{-1}\mathbf{W}_{k}
\right.\right.\nonumber\\
& & \hspace*{-1cm}\left.\left.
\hspace*{1.5cm}-\mathbf{W}^{T}_{k}\boldsymbol{\Lambda}_{k}^{-1}
\mathbf{X}(\mathbf{X}^{T}\boldsymbol{\Lambda}_{k}^{-1}\mathbf{X})^{-1}
\mathbf{X}^{T}\boldsymbol{\Lambda}_{k}^{-1}\mathbf{W}_{k}\right)\boldsymbol{\Lambda}_{k}^{1/2}\right.\nonumber\\
& & \hspace*{0.5cm}\left.\times \left(\mathbf{I}_{n\times
n}-2i\omega
\boldsymbol{\Lambda}_{k}^{1/2}\left(\mathbf{W}^{T}_{k}\boldsymbol{\Lambda}_{k}^{-1}\mathbf{W}_{k}\right.\right.\right.
\nonumber\\
& & \hspace*{0.5cm}\left.\left.\left.
-\mathbf{W}^{T}_{k}\boldsymbol{\Lambda}_{k}^{-1}
\mathbf{X}(\mathbf{X}^{T}\boldsymbol{\Lambda}_{k}^{-1}\mathbf{X})^{-1}
\mathbf{X}^{T}\boldsymbol{\Lambda}_{k}^{-1}\mathbf{W}_{k}\right)\boldsymbol{\Lambda}_{k}^{1/2}\right)^{-1}
\right.\nonumber\\
& &\left.\times
\boldsymbol{\Lambda}_{k}^{1/2}\left(\mathbf{W}^{T}_{k}\boldsymbol{\Lambda}_{k}^{-1}\mathbf{W}_{k}-\mathbf{W}^{T}_{k}\boldsymbol{\Lambda}_{k}^{-1}
\mathbf{X}(\mathbf{X}^{T}\boldsymbol{\Lambda}_{k}^{-1}\mathbf{X})^{-1}
\mathbf{X}^{T}\boldsymbol{\Lambda}_{k}^{-1}\mathbf{W}_{k}\right)\mathbf{X}\boldsymbol{\beta
}_{k}\right)\nonumber
\end{eqnarray}
\begin{eqnarray}
& &\times  \exp\left( i\omega\sum_{k=1}^{\infty}\boldsymbol{\beta
}_{k}^{T}\mathbf{X}^{T}\left(\mathbf{W}^{T}_{k}\boldsymbol{\Lambda}_{k}^{-1}\mathbf{W}_{k}-\mathbf{W}^{T}_{k}\boldsymbol{\Lambda}_{k}^{-1}
\mathbf{X}(\mathbf{X}^{T}\boldsymbol{\Lambda}_{k}^{-1}\mathbf{X})^{-1}\right.\right.\nonumber\\
& &\hspace*{4cm}\left.\left.\times
\mathbf{X}^{T}\boldsymbol{\Lambda}_{k}^{-1}\mathbf{W}_{k}\right)\mathbf{X}\boldsymbol{\beta
}_{k} \right).\nonumber\\
 \label{mgfSSE}
\end{eqnarray}

\end{itemize}
\end{theorem}

The proof can be found  in Appendix D.

\section{Linear functional hypothesis testing} \label{FLHT}
Consider the null  hypothesis
$$H_{0}:\
\mathbf{K}\boldsymbol{\beta }= \mathbf{C},$$ \noindent where
$\mathbf{C}\in H^{m},$ and  $\mathbf{K}$ is an matrix operator  from
$H^{p}$ into $H^{m}$ satisfying
\begin{eqnarray}
\hspace*{-1cm}\mathbf{K}&=&\left[\begin{array}{ccc} K_{11}& \dots & K_{1p}\\
\dots &\dots &\dots\\
K_{m1} &\dots & K_{mp}\\
\end{array}\right],\ K_{ij}(f)(g)=\sum_{k=1}^{\infty }\lambda_{k}(K_{ij})
\left\langle
\phi_{k},g\right\rangle_{H}\left\langle\phi_{k},f\right\rangle_{H}
\label{diagtestop}
\end{eqnarray}
\noindent for $f,g\in H,$ and for $i=1,\dots,m,$ and $j=1,\dots,p,$ in terms
of the eigenvectors $\phi_{k},$ $k\geq 1,$ of the functional entries
of $\mathbf{R}_{\boldsymbol{\varepsilon}\boldsymbol{\varepsilon}}.$
Thus, we restrict out attention to test some contrasts on the
functional components of $\boldsymbol{\beta},$ in terms of the class
of matrix operators  $\mathbf{K}$ such that, for each $l\geq 1,$
$\Phi_{l}^{*}\mathbf{K}\Phi_{l}=\mathbf{K}_{l},$
\begin{equation}
\mathbf{K}_{l}=\left[\begin{array}{ccc}
\lambda_{l}(K_{11}) & \dots & \lambda_{l}(K_{1p})\\
\dots & \dots & \dots \\
\lambda_{l}(K_{m1}) & \dots & \lambda_{l}(K_{mp})\\
\end{array}\right]_{m\times p}.\label{eqklc}
\end{equation}

\begin{remark}
Since the developed Functional Analysis of Variance in a
multivariate context is referred to the orthogonal basis of $H,$
$\{\phi_{k},\ k\geq 1\},$  which is assumed to be known (see Remark
\ref{kev} and Section \ref{Exm}), a natural way of defining possible
linear transformations $\mathbf{K}$ of our functional parameter
vector $\boldsymbol{\beta}\in H^{p},$ to test some contrasts, is
given by  (\ref{eqklc}), with
$\Phi_{l}^{*}\mathbf{K}\Phi_{l}=\mathbf{K}_{l},$ for   $l\geq 1.$
  \end{remark}
\begin{lemma}
\label{lemHT1} The generalized least squares estimator
$\widehat{\boldsymbol{\beta }}$ defined in equations
(\ref{lse1})--(\ref{solutlsep}) satisfies
$\Phi_{k}^{*}(\widehat{\boldsymbol{\beta } })\sim
\mathcal{N}(\Phi_{k}^{*}(\boldsymbol{\beta
}),(\mathbf{X}^{T}\boldsymbol{\Lambda}_{k}^{-1}\mathbf{X})^{-1}),$
for each $k\geq 1.$  Under condition (\ref{dmsc}), this estimator is
a Hilbert-valued Gaussian random variable ($H^{p}$-valued random
variable), with functional mean $\boldsymbol{\beta }$ and trace
covariance operator $\mathbf{Q}$ such that, for each $k\geq 1,$
$\Phi_{k}^{*}\mathbf{Q}\Phi_{k}=(\mathbf{X}^{T}\boldsymbol{\Lambda}_{k}^{-1}\mathbf{X})^{-1}.$
Equivalently, $$\mathbf{Q}=\left[\begin{array}{cccc}
\sum_{k=1}^{\infty }q_{k1}\phi_{k}\otimes \phi_{k} & \dots & \dots &
\sum_{k=1}^{\infty
}q_{k1p}\phi_{k}\otimes \phi_{k}\\
\vdots & \vdots & \vdots & \vdots \\
\vdots & \vdots & \vdots & \vdots \\
\sum_{k=1}^{\infty }q_{kp1}\phi_{k}\otimes \phi_{k} & \dots & \dots
& \sum_{k=1}^{\infty
}q_{kp}\phi_{k}\otimes \phi_{k} \\
\end{array}
\right],$$ \noindent where,  for each $k\geq 1,$ $q_{ki}$ represents
the $i$th diagonal element of matrix
$(\mathbf{X}^{T}\boldsymbol{\Lambda}_{k}^{-1}\mathbf{X})^{-1},$
$i=1,\dots,p,$  and $q_{kij}$ and $q_{kji},$ with $q_{kij}=q_{kji},$
denote its $(i,j)$ and $(j,i)$ entries, $i\neq j,$ and
$i,j=1,\dots,p.$
\end{lemma}
The proof directly follows from classical generalized least-squares
theory (see, for example,   \cite{Hocking}  and \cite{Ruud}), and from
condition (\ref{dmsc}) (see, for example, \cite{Da Prato}, pp.8-17,
Chapter 1).

Note that
 $\Phi^{*}_{l}(\mathbf{C})=\mathbf{C}_{l}$
represents a $m\times 1$ vector for each $l\geq 1.$ Hence, for each
$k\geq 1,$ $\Phi_{k}^{*}(\mathbf{K}\boldsymbol{\beta
})=[\mathbf{K}\boldsymbol{\beta }]_{k}=\mathbf{C}_{k}.$
As a direct consequence of Lemma \ref{lemHT1}, the following result
is considered (see, for example,  \cite{Hocking} and \cite{Ruud}).
\begin{lemma}
\label{lem2HT} Under the null hypothesis
$H_{0}:\mathbf{K}\boldsymbol{\beta }= \mathbf{C},$ for each $l\geq
1,$
$$[\Phi_{l}^{*}\mathbf{K}\Phi_{l}\Phi_{l}^{*}\widehat{\boldsymbol{\beta
}}-\Phi_{l}^{*}\mathbf{C}]^{T}[\Phi_{l}^{*}\mathbf{K}\Phi_{l}
(\mathbf{X}^{T}\boldsymbol{\Lambda}_{l}^{-1}\mathbf{X})^{-1}
[\Phi_{l}^{*}\mathbf{K}\Phi_{l}]^{T}]^{-1}\Phi_{l}^{*}\mathbf{K}\Phi_{l}\Phi_{l}^{*}\widehat{\boldsymbol{\beta
}}-\Phi_{l}^{*}\mathbf{C},$$ \noindent follows a chi-squared
distribution with $m$ degrees of freedom.
\end{lemma}
Lemma \ref{lem2HT} allows the application of an extended version of the finite-dimensional
 hipothesis testing  approach for functional data proposed in
\cite{Cuesta-Albertos10}, based on the formulation of Cram\'er-Wold
theorem derived in \cite{Cuesta-Albertos07}. In our multivariate infinite-dimensional  one-way  ANOVA,   the random
vectors should be selected from a Gaussian distribution on the Hilbert space
$\widetilde{\mathcal{H}}=H^{m}$ with non-degenerate $m$-dimensional
projections.

The next proposition allows the application of Lemma \ref{lem1} for the formulation of a functional linear test.

\begin{proposition}
\label{lhtf} Assume that condition (\ref{dmsc}) holds, and that
matrix operator   $\mathbf{K}$   is such that
\begin{equation}
\|(\mathbf{X}^{T}\mathbf{R}_{\boldsymbol{\varepsilon}\boldsymbol{\varepsilon}}^{-1}\mathbf{X})^{-1/2}
\mathbf{K}^{T}\mathbf{K}(\mathbf{X}^{T}\mathbf{R}_{\boldsymbol{\varepsilon}\boldsymbol{\varepsilon}}^{-1}
\mathbf{X})^{-1/2}\|_{\mathcal{L}(H^{n})=\mathcal{L}(\mathcal{H})}<1.\label{eqcKt}
\end{equation}
\noindent Then, under the null hypothesis
$H_{0}:\mathbf{K}\boldsymbol{\beta }= \mathbf{C},$ the test
statistic $$\left\langle \mathbf{K}\widehat{\boldsymbol{\beta
}}-\mathbf{C},\mathbf{K}\widehat{\boldsymbol{\beta
}}-\mathbf{C}\right\rangle_{\mathcal{H}=H^{n}}=\sum_{l=1}^{\infty
}(\widehat{\boldsymbol{\beta }}_{l}-\boldsymbol{\beta
}_{l})^{T}\mathbf{K}_{l}^{T}\mathbf{K}_{l}(\widehat{\boldsymbol{\beta
}}_{l}-\boldsymbol{\beta }_{l}),$$ \noindent  has
 characteristic functional
given by
\begin{eqnarray}& & E\left[\exp\left(i\omega\left\langle \mathbf{K}\widehat{\boldsymbol{\beta
}}-\mathbf{C},\mathbf{K}\widehat{\boldsymbol{\beta
}}-\mathbf{C}\right\rangle_{\mathcal{H}=H^{n}}\right)\right]\nonumber\\
& & =E\left[\exp\left(i\omega\sum_{l=1}^{\infty
}(\widehat{\boldsymbol{\beta }}_{l}-\boldsymbol{\beta
}_{l})^{T}\mathbf{K}_{l}^{T}\mathbf{K}_{l}(\widehat{\boldsymbol{\beta
}}_{l}-\boldsymbol{\beta
}_{l})\right)\right]\nonumber\\
& &= \prod_{l=1}^{\infty }\left[\mbox{det}\left(\mathbf{I}_{p\times
p}-2i\omega (\mathbf{X}^{T}\boldsymbol{\Lambda
}_{l}^{-1}\mathbf{X})^{-1/2}
\mathbf{K}^{T}_{l}\mathbf{K}_{l}(\mathbf{X}^{T}\boldsymbol{\Lambda
}_{l}^{-1} \mathbf{X})^{-1/2}\right)\right]^{-1/2}.
\label{chprchlht}
\end{eqnarray}
\noindent Here, as before, $\widehat{\boldsymbol{\beta }}_{l}=
\Phi_{l}^{*}(\widehat{\boldsymbol{\beta }}),$ and
$\Phi_{l}^{*}\mathbf{K}\Phi_{l}=\mathbf{K}_{l},$ for each $l\geq 1.$
\end{proposition}
The proof of Proposition \ref{lhtf}  directly follows from Lemmas \ref{lem1}
and \ref{lemHT1}, under condition (\ref{eqcKt}), considering $\mathbf{M}=2i\omega
\mathbf{K}^{*}\mathbf{K}$ and $\mathbf{Y}=\widehat{\boldsymbol{\beta
}}-\boldsymbol{\beta }$ in Lemma \ref{lem1}.
\begin{theorem}
\label{th7} Assume that the conditions considered in Proposition
\ref{lhtf} hold. Then, for testing
$H_{0}:\mathbf{K}\boldsymbol{\beta }=\mathbf{C}$ versus
$H_{1}:\mathbf{K}\boldsymbol{\beta }\neq \mathbf{C},$ at level
$\alpha ,$ there exists a test $\psi $ given by:
$$\psi=\left\{\begin{array}{l}
1\quad \mbox{if}\quad S_{H_{0}}(\mathbf{Y})>C(H_{0},\alpha ),\\
0\quad \mbox{otherwise.}\\
\end{array}\right.$$
\noindent Here, $S_{H_{0}}(\mathbf{Y})=\left\langle \mathbf{K}\widehat{\boldsymbol{\beta
}}-\mathbf{C},\mathbf{K}\widehat{\boldsymbol{\beta
}}-\mathbf{C}\right\rangle_{\mathcal{H}=H^{n}}=\sum_{l=1}^{\infty
}(\widehat{\boldsymbol{\beta }}_{l}-\boldsymbol{\beta
}_{l})^{T}\mathbf{K}_{l}^{T}\mathbf{K}_{l}(\widehat{\boldsymbol{\beta
}}_{l}-\boldsymbol{\beta }_{l}).$ The constant $C(H_{0},\alpha )$
is such that \begin{eqnarray}\mathbb{P}\left\{ S_{H_{0}}(\mathbf{Y})>C(H_{0},\alpha
),\mathbf{K}\boldsymbol{\beta
}=\mathbf{C}\right\}&=&1-\mathbb{P}\left\{ S_{H_{0}}(\mathbf{Y})\leq
C(H_{0},\alpha ),\mathbf{K}\boldsymbol{\beta
}=\mathbf{C}\right\}\nonumber\\
&=& 1-\mathbf{F}_{\alpha }=\alpha ,\nonumber
\end{eqnarray}
 \noindent where the
probability distribution  $\mathbf{F}$ on $\mathcal{H}=H^{n}$ has
characteristic functional given in equation (\ref{chprchlht}) of
Proposition \ref{lhtf}.
\end{theorem}

\section{Final comments}
\label{Sec6}

  Our
approach allows the identification of the functional components of
variance $\widetilde{\mbox{\textbf{SST}}},$
$\widetilde{\mbox{\textbf{SSR}}}$ and
$\widetilde{\mbox{\textbf{SSE}}}$ with  series of independent
finite-dimensional random quadratic forms, respectively constructed
from a sequence of  independent multivariate ($n$-dimensional)
Gaussian random variables (see Theorem \ref{characteristicfunctions}). The elements of such a multivariate
normal sequence respectively define the $n$-dimensional projections
of the infinite-dimensional  multivariate Gaussian measure on
$\mathcal{H}=H^{n},$ underlying to our studied Hilbert-valued
fixed-effect Gaussian model with correlated  error components. New
functional hypothesis tests, in the spirit of classical one way
F-ANOVA test, can be formulated from such identifications, applying
an extended version of the methodology proposed in
\cite{Cuesta-Albertos10} for independent
functional data.

The example given in Section \ref{Exm} can be
applied to different practical si-\linebreak tuations. Specifically,
in the geophysical context, the response can be referred to time
(respectively, to space), i.e., the
 response  takes its values in a separable Hilbert space
of vector functions with temporal support (respectively, with spatial
support). In such a case, equation (\ref{sdeex}) represents the
physical law governing the evolution in time of the
vectorial random source or multivariate innovation process (respectively, describing the
spatial diffusion of the vectorial random source or multivariate spatial innovation), at
each station  (for example, heat transfer  equation at each
meteorological station, under different climatological conditions,
e.g., considering seasonal factors, location factors, etc.),
 in the heterocedastic
and correlated settings   (see, for instance, \cite{Wikle98}, for a
mixed-effect spatiotemporal process formulation applied to
meteorology). In human tactile perception (see, for example,
\cite{Spitzner}), for $i=1,\dots,n,$ operator $f_{i}(\mathcal{L})$
can define the movement equation describing  the  path of the random
stimulus applied to  the $i$th subject, under different experimental
conditions, which can affect the perception of the subject.  Here,
the applied random stimuli interact or are correlated between
different individuals. Other possible fields of application of the
proposed Hilbert-valued Gaussian fixed effect model with functional
correlated noise can be found in the statistical analysis of
functional magnetic resonance imaging data (see, for example,
\cite{Kang12}, and the references therein), disease mapping (see,
for example, \cite{Ruiz_Medina13}), and in  spatiotemporal environmental processes (see
\cite{Wikle03}).

Alternatively, the semi-parametric class of covariance matrix
operators introduced in (\ref{covopsd}) can be reformulated in a
more flexible way,  by considering an extended definition  of the
matrix sequence $\{\boldsymbol{\Lambda }_{k},\ k\geq 1\},$ given in
(\ref{eqprocovop}). Specifically, in (\ref{eqprocovop}), a separable
correlation structure between random variables $\{\eta_{ki},\ k\geq
1\}$ and $\{\eta_{kj},\ k\geq 1\},$ for $i\neq j,$ and $i,j\in
\{1,\dots,n\}$ is considered, ensuring, from equations
(\ref{SI})--(\ref{tracecvo}),  that
$\mbox{trace}\left(\sum_{k=1}^{\infty
}\boldsymbol{\Lambda}_{k}\right)<\infty.$ In a more general
framework, we can consider, for $i\neq j,$
\begin{equation}E[\eta_{ki}\eta_{pj}]=\delta_{k,p}f(\lambda_{ki},\lambda_{kj},i,j),\label{eqffunct0}
\end{equation}
\noindent with $f$ being a  function such that the matrix
$\boldsymbol{\Lambda }_{k}$ is of full rank, for every $k\geq 1,$
and condition  (\ref{tracecvo}) is satisfied.

Extensions of the formulated results to the framework of
Hilbert-valued fixed effect models with autoregressive correlated
error components can be obtained from the presented methodology and the results derived in \cite{Bosq} and
\cite{BosqBlanke}, in the temporal autoregressive case, as well as
the results given  in \cite{Ruiz_Medina11} and \cite{Ruiz_Medina12},
in the spatial autoregressive case. However, in the $H$-valued multivariate time series
framework, the assumption on the existence of a common resolution of the identity could exclude some
interesting cases.  Hence, further research is required to obtain a more flexible setting of assumptions.
\newpage
\noindent {\bfseries {\large Acknowledgments}}

This work has been supported in part by project MTM2012-32674 (cofunded with FEDER) of the
DGI, MEC,  Spain


\section*{Appendix}
In the following, Appendices A--C provide the proof of the almost surely
finiteness of $\widetilde{\mbox{\textbf{SST}}},$
$\widetilde{\mbox{\textbf{SSR}}}$ and
$\widetilde{\mbox{\textbf{SSE}}}.$ In addition, Appendix D gives the proof of Theorem \ref{characteristicfunctions}, and Appendix E shows
an example of the Hilbert space structures that can be considered under {\bfseries Assumption A0}.
\bigskip

\subsection*{\bfseries Appendix A. Almost surely finiteness of $\widetilde{\mbox{\textbf{SST}}}$}

\medskip

\subsubsection*{Proof of Proposition \ref{SST}}

 \textcolor{blue}{Let us denote $\mathcal{T}(\mathbf{W}_{k},\boldsymbol{\Lambda}_{k})= \mbox{trace}\left(\mathbf{W}_{k}^{T}\boldsymbol{\Lambda}_{k}^{-1}\mathbf{W}_{k}\right)$
 and $\mathcal{T}(\boldsymbol{\Lambda}_{k})=\mbox{trace}\left(
\boldsymbol{\Lambda}_{k}\right).$ Also,  the notation $\| \mathbf{X}\boldsymbol{\beta
} \|^{2}_{\mathbf{W}R^{-1}_{\boldsymbol{\varepsilon}\boldsymbol{\varepsilon}}\mathbf{W}}=\sum_{k=1}^{\infty }\boldsymbol{\beta
}^{T}_{k}\mathbf{X}^{T}\mathbf{W}_{k}^{T}\boldsymbol{\Lambda}_{k}^{-1}\mathbf{W}_{k}\mathbf{X}\boldsymbol{\beta
}_{k}$ will be used.}
\textcolor{blue}{
 Applying Cauchy–Schwarz
inequality, and Parselval identity, under (\ref{c1asSST}) (respectively, (\ref{c2asSST})), we
obtain
\begin{eqnarray}
& & \hspace*{-0.75cm} E[\widetilde{\mbox{\textbf{SST}}}] =\sum_{k=1}^{\infty }
\mbox{trace}\left(\mathbf{W}_{k}^{T}\boldsymbol{\Lambda}_{k}^{-1}\mathbf{W}_{k}\boldsymbol{\Lambda}_{k}\right)+\| \mathbf{X}\boldsymbol{\beta
} \|^{2}_{\mathbf{W}R^{-1}_{\boldsymbol{\varepsilon}\boldsymbol{\varepsilon}}\mathbf{W}}\leq \sum_{k=1}^{\infty }
\mathcal{T}(\mathbf{W}_{k},\boldsymbol{\Lambda}_{k})\mathcal{T}(\boldsymbol{\Lambda}_{k})\nonumber\\
& & \hspace*{-0.75cm}+\| \mathbf{X}\boldsymbol{\beta
} \|^{2}_{\mathbf{W}R^{-1}_{\boldsymbol{\varepsilon}\boldsymbol{\varepsilon}}\mathbf{W}}\leq \sqrt{\sum_{k=1}^{\infty } \left[\mathcal{T}(\boldsymbol{\Lambda}_{k})\right]^{2}}\sqrt{\sum_{k=1}^{\infty
}\left[\mathcal{T}\left(\mathbf{W}_{k},\boldsymbol{\Lambda}_{k}\right)\right]^{2}}+\| \mathbf{X}\boldsymbol{\beta
} \|^{2}_{\mathbf{W}R^{-1}_{\boldsymbol{\varepsilon}\boldsymbol{\varepsilon}}\mathbf{W}}
\nonumber\\
& &\hspace*{-0.75cm}
\leq \sqrt{\sum_{k=1}^{\infty } \left[\mathcal{T}(\boldsymbol{\Lambda}_{k})\right]^{2}} \sqrt{\sum_{k=1}^{\infty
}\left[\mathcal{T}(\mathbf{W}_{k},\boldsymbol{\Lambda}_{k})\right]^{2}} +  \sum_{k=1}^{\infty
}[\mathcal{T}(\mathbf{W}_{k},\boldsymbol{\Lambda}_{k})]^{2}
 \boldsymbol{\beta
}^{T}_{k}\mathbf{X}^{T}\mathbf{X}\boldsymbol{\beta
}_{k}
\nonumber\\
&\leq & \sqrt{\sum_{k=1}^{\infty }
\left[\mathcal{T}(\boldsymbol{\Lambda}_{k})\right]^{2}\sum_{k=1}^{\infty
}\left[\mathcal{T}(\mathbf{W}_{k},\boldsymbol{\Lambda}_{k})\right]^{2}}+ \sqrt{\sum_{k=1}^{\infty
}[\mathcal{T}(\mathbf{W}_{k},\boldsymbol{\Lambda}_{k})]^{4}
\sum_{k=1}^{\infty
}\left[\boldsymbol{\beta
}^{T}_{k}\mathbf{X}^{T}\mathbf{X}\boldsymbol{\beta
}_{k}\right]^{2}},
\nonumber
\end{eqnarray}
}

\noindent \textcolor{blue}{which finite, since, under (\ref{c1asSST}) (respectively, (\ref{c2asSST})), $\sum_{k=1}^{\infty}[\mathcal{T}(\mathbf{W}_{k},\boldsymbol{\Lambda}_{k})]^{4} \leq
C\sum_{k=1}^{\infty}\mathcal{T}(\mathbf{W}_{k},\boldsymbol{\Lambda}_{k})<\infty ,$
 and since $\mathbf{X}\boldsymbol{\beta
}\in \mathcal{H}=H^{n},$ then $\sum_{k=1}^{\infty
}\left[\boldsymbol{\beta
}^{T}_{k}\mathbf{X}^{T}\mathbf{X}\boldsymbol{\beta
}_{k}\right]^{2}\leq \widetilde{C}\|\mathbf{X}\boldsymbol{\beta
}\|_{\mathcal{H}=H^{n}}^{2}=\widetilde{C}\sum_{k=1}^{\infty}\boldsymbol{\beta
}^{T}_{k}\mathbf{X}^{T}\mathbf{X}\boldsymbol{\beta
}_{k}<\infty ,$ for certain positive constants $C$ and $\widetilde{C}.$}

\subsection*{\bfseries Appendix B. Almost surely finiteness of
$\widetilde{\mbox{\textbf{SSR}}}$}

\subsubsection*{Proof of Proposition \ref{SSR}}

\textcolor{blue}{Let us compute
\begin{eqnarray} && \hspace*{-0.75cm} E[\widetilde{\mbox{\textbf{SSR}}}]=\sum_{k=1}^{\infty
}E\left[\mathbf{Y}_{k}^{T}[\mathbf{W}_{k}^{T}\boldsymbol{\Lambda}_{k}^{-1}\mathbf{W}_{k}-\mathbf{W}_{k}^{T}\boldsymbol{\mathcal{M}}_{k}^{T}
\boldsymbol{\Lambda}_{k}^{-1}
\boldsymbol{\mathcal{M}}_{k}\mathbf{W}_{k}]\mathbf{Y}_{k}\right]=E\left[\widetilde{\mbox{\textbf{SST}}}\right]\nonumber\\
& &\hspace*{-0.75cm}-E\left[\|\mathbf{Y}\|^{2}_{\mathbf{W}\boldsymbol{\mathcal{M}}
R^{-1}_{\boldsymbol{\varepsilon}\boldsymbol{\varepsilon}}\boldsymbol{\mathcal{M}}\mathbf{W}}
\right]=E\left[\widetilde{\mbox{\textbf{SST}}}\right]-\mathcal{T}\left(\mathbf{W}\boldsymbol{\mathcal{M}}^{T}R^{-1}_{\boldsymbol{\varepsilon}
\boldsymbol{\varepsilon}}\boldsymbol{\mathcal{M}}\mathbf{W}
R_{\boldsymbol{\varepsilon}\boldsymbol{\varepsilon}}\right)-\|\mathbf{X}\boldsymbol{\beta}\|^{2}_{\mathbf{W}
\boldsymbol{\mathcal{M}}^{T}R^{-1}_{\boldsymbol{\varepsilon}\boldsymbol{\varepsilon}}\boldsymbol{\mathcal{M}}\mathbf{W}}
\nonumber\\
&&\hspace*{-0.75cm}
 =E\left[\widetilde{\mbox{\textbf{SST}}}\right]-\mathcal{T}\left(\mathbf{W}, \boldsymbol{\Lambda
}^{-1},\boldsymbol{\Lambda
}\right)
+\mathcal{T}\left(\mathbf{W},\mathbf{X}, \boldsymbol{\Lambda
}^{-1},\boldsymbol{\Lambda
}\right)-\|\mathbf{X}\boldsymbol{\beta}\|^{2}_{\mathbf{W}\boldsymbol{\mathcal{M}}^{T}R^{-1}_{\boldsymbol{\varepsilon}\boldsymbol{\varepsilon}}
\boldsymbol{\mathcal{M}}\mathbf{W}}.
 \label{ssr22b}
\end{eqnarray}}
\noindent \textcolor{blue}{where $E\left[\|\mathbf{Y}\|^{2}_{\mathbf{W}\boldsymbol{\mathcal{M}}R^{-1}_{\boldsymbol{\varepsilon}\boldsymbol{\varepsilon}}\boldsymbol{\mathcal{M}}\mathbf{W}}
\right]=\sum_{k=1}^{\infty
}E\left[
\mathbf{Y}_{k}^{T}\mathbf{W}_{k}^{T}\boldsymbol{\mathcal{M}}_{k}^{T}\boldsymbol{\Lambda}_{k}^{-1}
\boldsymbol{\mathcal{M}}_{k}\mathbf{W}_{k}\mathbf{Y}_{k}\right],$  \\ $\mathcal{T}\left(\mathbf{W}\boldsymbol{\mathcal{M}}^{T}R^{-1}_{\boldsymbol{\varepsilon}\boldsymbol{\varepsilon}}\boldsymbol{\mathcal{M}}\mathbf{W}
R_{\boldsymbol{\varepsilon}\boldsymbol{\varepsilon}}\right)=\sum_{k=1}^{\infty
}\mbox{trace}\left(\mathbf{W}_{k}^{T}\boldsymbol{\mathcal{M}}_{k}^{T}\boldsymbol{\Lambda}_{k}^{-1}
\boldsymbol{\mathcal{M}}_{k}\mathbf{W}_{k}\boldsymbol{\Lambda}_{k}\right),$  \linebreak $\|\mathbf{X}\boldsymbol{\beta}\|^{2}_{\mathbf{W}\boldsymbol{\mathcal{M}}^{T}R^{-1}_{\boldsymbol{\varepsilon}\boldsymbol{\varepsilon}}
\boldsymbol{\mathcal{M}}\mathbf{W}}
=\sum_{k=1}^{\infty }
\boldsymbol{\beta}_{k}\mathbf{X}^{T}\mathbf{W}_{k}^{T}\boldsymbol{\mathcal{M}}_{k}^{T}\boldsymbol{\Lambda}_{k}^{-1}
\boldsymbol{\mathcal{M}}_{k}\mathbf{W}_{k}\mathbf{X}\boldsymbol{\beta}_{k},$ $\mathcal{T}\left(\mathbf{W}, \boldsymbol{\Lambda
}^{-1},\boldsymbol{\Lambda
}\right)=\sum_{k=1}^{\infty
}\mbox{trace}\left(\mathbf{W}_{k}^{T}\boldsymbol{\Lambda
}_{k}^{-1}\mathbf{W}_{k}\boldsymbol{\Lambda }_{k}\right),$ and $\mathcal{T}\left(\mathbf{W}, \mathbf{X},\boldsymbol{\Lambda
}^{-1},\boldsymbol{\Lambda
}\right)=$\\$\sum_{k=1}^{\infty
}\mbox{trace}\left(\mathbf{W}_{k}^{T}\boldsymbol{\Lambda
}_{k}^{-1}\mathbf{X}(\mathbf{X}^{T}\boldsymbol{\Lambda
}_{k}^{-1}\mathbf{X})^{-1}\mathbf{X}^{T}\boldsymbol{\Lambda
}_{k}^{-1}\mathbf{W}_{k}\boldsymbol{\Lambda
}_{k}\right)$}

 Note that, under condition
(\ref{c1asSST2}), since $\mathbf{W}_{k}=\mathbf{W}_{k}^{T},$
$\sum_{k=1}^{\infty
}\mbox{trace}\left(\mathbf{W}_{k}^{1/2}\boldsymbol{\Lambda
}_{k}^{-1}\mathbf{W}_{k}^{1/2}\right)$ is finite.
Moreover, applying again well-known properties of the trace of the
product of real and symmetric  positive semi-definite  matrices
(see, for example, \cite{FangY}), and Cauchy-Schwarz inequality, we
obtain
\textcolor{blue}{\begin{eqnarray}& & \mathcal{T}\left(\mathbf{W}, \mathbf{X},\boldsymbol{\Lambda
}^{-1},\boldsymbol{\Lambda
}\right) \leq \widetilde{\mathcal{T}}\left(\mathbf{W}, \mathbf{X},\boldsymbol{\Lambda
}^{-1},\boldsymbol{\Lambda
}\right)
\nonumber\\
& & \leq \sum_{k=1}^{\infty }\mbox{trace}\left(
(\mathbf{X}^{T}\boldsymbol{\Lambda
}_{k}^{-1}\mathbf{X})^{-1}\right)\mbox{trace}\left(\mathbf{X}^{T}\boldsymbol{\Lambda
}_{k}^{-1}\mathbf{W}_{k}\mathbf{W}_{k}^{T}\boldsymbol{\Lambda
}_{k}^{-1}\mathbf{X}\right)\mbox{trace}\left(\boldsymbol{\Lambda
}_{k}\right)\nonumber\\
 & &\leq \sum_{k=1}^{\infty }\mbox{trace}\left(
(\mathbf{X}^{T}\boldsymbol{\Lambda }_{k}^{-1}\mathbf{X})^{-1}\right)
[\mbox{trace}\left(\boldsymbol{\Lambda
}_{k}^{-1}\mathbf{W}_{k}\right)]^{2}\mathcal{T}\left(\mathbf{X}\mathbf{X}^{T}\right)\mathcal{T}\left(\boldsymbol{\Lambda }_{k}\right)\nonumber\\
& &\leq
\mathcal{T}\left(\mathbf{X}\mathbf{X}^{T}\right)\sqrt{\sum_{k=1}^{\infty
}\left[\mbox{trace}\left(\boldsymbol{\Lambda
}_{k}\right)\mbox{trace}\left( (\mathbf{X}^{T}\boldsymbol{\Lambda
}_{k}^{-1}\mathbf{X})^{-1}\right)\right]^{2}\sum_{k=1}^{\infty
}\left[\mbox{trace}\left(\mathbf{W}_{k}^{1/2}\boldsymbol{\Lambda
}_{k}^{-1}\mathbf{W}_{k}^{1/2}\right)\right]^{4}}\nonumber\\
& &\leq
\mathcal{K}\mathcal{T}\left(\mathbf{X}\mathbf{X}^{T}\right)\sqrt{\sum_{k=1}^{\infty
}\mathcal{T}\left(\boldsymbol{\Lambda
}_{k}\right)\mbox{trace}\left( (\mathbf{X}^{T}\boldsymbol{\Lambda
}_{k}^{-1}\mathbf{X})^{-1}\right)\mathcal{T}\left(\mathbf{W}^{1/2},\boldsymbol{\Lambda
}^{-1}\right)}  \label{ssr22c}
\end{eqnarray}}
\textcolor{blue}{\begin{eqnarray}
& &\leq
\mathcal{K}
\mathcal{T}\left(\mathbf{X}\mathbf{X}^{T}\right)\left[\sum_{k=1}^{\infty
}\left[\mathcal{T}\left(\boldsymbol{\Lambda
}_{k}\right)\right]^{2}\mathcal{T}\left(\mathbf{X},\boldsymbol{\Lambda}^{-1}\right)\right]^{1/4} \sqrt{\mathcal{T}\left(\mathbf{W}^{1/2},\boldsymbol{\Lambda
}^{-1}\right)}<\infty,\nonumber
\end{eqnarray}}
\noindent    under conditions (\ref{dmsc}) and
(\ref{c1asSST2}),
 where \textcolor{blue}{$\mathcal{T}\left(\mathbf{W}, \mathbf{X},\boldsymbol{\Lambda
}^{-1},\boldsymbol{\Lambda
}\right)$\\ $=\sum_{k=1}^{\infty
}\mbox{trace}\left(\mathbf{W}_{k}^{T}\boldsymbol{\Lambda
}_{k}^{-1}\mathbf{X}(\mathbf{X}^{T}\boldsymbol{\Lambda
}_{k}^{-1}\mathbf{X})^{-1}\mathbf{X}^{T}\boldsymbol{\Lambda
}_{k}^{-1}\mathbf{W}_{k}\boldsymbol{\Lambda }_{k}\right),$
$\widetilde{\mathcal{T}}\left(\mathbf{W}, \mathbf{X},\boldsymbol{\Lambda
}^{-1},\boldsymbol{\Lambda
}\right)=
\sum_{k=1}^{\infty
}\mbox{trace}\left(\mathbf{W}_{k}^{T}\boldsymbol{\Lambda
}_{k}^{-1}\mathbf{X}(\mathbf{X}^{T}\boldsymbol{\Lambda
}_{k}^{-1}\mathbf{X})^{-1}\mathbf{X}^{T}\boldsymbol{\Lambda
}_{k}^{-1}\mathbf{W}_{k}\right)\mbox{trace}\left(\boldsymbol{\Lambda }_{k}\right),$\\
$\mathcal{T}\left(\mathbf{X}\mathbf{X}^{T}\right)=\mbox{trace}\left(\mathbf{X}\mathbf{X}^{T}\right),$
$\mathcal{T}\left(\mathbf{X},\boldsymbol{\Lambda}^{-1}\right)=\sum_{k=1}^{\infty
}\left[\mbox{trace}\left( (\mathbf{X}^{T}\boldsymbol{\Lambda
}_{k}^{-1}\mathbf{X})^{-1}\right)\right]^{2},$ \\ and $\mathcal{T}\left(\mathbf{W}^{1/2},\boldsymbol{\Lambda
}^{-1}\right)=\sum_{k=1}^{\infty
}\mbox{trace}\left(\mathbf{W}_{k}^{1/2}\boldsymbol{\Lambda
}_{k}^{-1}\mathbf{W}_{k}^{1/2}\right).$
}
In addition,
\textcolor{blue}{
\begin{eqnarray}& &
\sum_{k=1}^{\infty
}\boldsymbol{\beta}_{k}\mathbf{X}^{T}\mathbf{W}_{k}^{T}\boldsymbol{\mathcal{M}}_{k}^{T}\boldsymbol{\Lambda}_{k}^{-1}
\boldsymbol{\mathcal{M}}_{k}\mathbf{W}_{k}\mathbf{X}\boldsymbol{\beta}_{k}\leq \sum_{k=1}^{\infty
}\boldsymbol{\beta}_{k}\mathbf{X}^{T}\mathbf{W}_{k}^{T}\boldsymbol{\Lambda}_{k}^{-1}
\mathbf{W}_{k}\mathbf{X}\boldsymbol{\beta}_{k}\nonumber\\
& & \leq \sum_{k=1}^{\infty
}\left[\mbox{trace}\left(\mathbf{W}_{k}^{T}\boldsymbol{\Lambda}_{k}^{-1}\mathbf{W}_{k}\right)\right]^{2}
\boldsymbol{\beta}_{k}\mathbf{X}^{T}\mathbf{X}\boldsymbol{\beta}_{k} \leq \sqrt{\sum_{k=1}^{\infty
}\left[\mbox{trace}\left(\mathbf{W}_{k}^{T}\boldsymbol{\Lambda}_{k}^{-1}\mathbf{W}_{k}\right)\right]^{4}}
\|\mathbf{X}\boldsymbol{\beta}\|_{\mathcal{H}}\nonumber\\
& & \leq \widetilde{\mathcal{K}}\sqrt{\sum_{k=1}^{\infty
}\mbox{trace}\left(\mathbf{W}_{k}^{T}\boldsymbol{\Lambda}_{k}^{-1}\mathbf{W}_{k}\right)}
\|\mathbf{X}\boldsymbol{\beta}\|_{\mathcal{H}}<\infty,
 \label{ssr22d}
\end{eqnarray}}
\noindent since $\mathbf{X}\boldsymbol{\beta}\in \mathcal{H}=H^{n},$
and $\sum_{k=1}^{\infty
}\mbox{trace}\left(\mathbf{W}_{k}^{T}\boldsymbol{\Lambda}_{k}^{-1}\mathbf{W}_{k}\right)<\infty,$
under condition (\ref{c1asSST2}).  From equations
(\ref{ssr22b})--(\ref{ssr22d}), keeping in mind that,
 $E\left[\widetilde{\mbox{\textbf{SST}}}\right]<\infty,$ as given in Appendix A, we have
 $E[\widetilde{\mbox{\textbf{SSR}}}]<\infty .$ Thus,
 $\widetilde{\mbox{\textbf{SSR}}}$ is a.s. finite as we wanted to
 prove.

\subsection*{\bfseries Appendix C. Almost surely finiteness of
$\widetilde{\mbox{\textbf{SSE}}}$}

The a.s. finiteness of $\widetilde{\mbox{\textbf{SSE}}}$ follows
straightforward from Appendix B,
 since
\begin{eqnarray}
& & E[\widetilde{\mbox{\textbf{SSE}}}]= \sum_{k=1}^{\infty }E\left[
\mathbf{Y}_{k}^{T}\mathbf{W}_{k}^{T}\boldsymbol{\mathcal{M}}_{k}^{T}\boldsymbol{\Lambda}_{k}^{-1}
\boldsymbol{\mathcal{M}}_{k}\mathbf{W}_{k}\mathbf{Y}_{k}\right]<\infty,
\end{eqnarray}
\noindent as given in  equations (\ref{ssr22b})--(\ref{ssr22d}).
\subsection*{\bfseries Appendix D. Proof of Theorem \ref{characteristicfunctions}}

The proof of Theorem \ref{characteristicfunctions} follows from
Lemma \ref{lem1}. We now provide the main steps involved in the
  derivation of the characteristic functional
of the components of variance for the transformed functional data
model.

\begin{itemize}
\item[] (i)
From Lemma \ref{lem1}, considering as $H^{n}$-valued  zero-mean
Gaussian random variable,   $\mathbf{Y}-\mathbf{X}\boldsymbol{\beta
},$ and as matrix operator $\mathbf{M}=2i\omega
\mathbf{W}^{*}\mathbf{R}_{\boldsymbol{\varepsilon}\boldsymbol{\varepsilon}}^{-1}\mathbf{W},$
keeping in mind that
$\Phi_{k}^{*}(\mathbf{Y}-\mathbf{X}\boldsymbol{\beta
})=\mathbf{Y}_{k}-\mathbf{X}\boldsymbol{\beta }_{k},$ and
$\Phi_{k}^{*}\mathbf{W}^{*}\mathbf{R}_{\boldsymbol{\varepsilon}\boldsymbol{\varepsilon}}^{-1}\mathbf{W}\Phi_{k}=\mathbf{W}_{k}^{T}\boldsymbol{\Lambda}_{k}^{-1}\mathbf{W}_{k},$
for each $k\geq 1,$ we obtain, for   $\omega
<B_{\widetilde{\mbox{\textbf{SST}}}},$ with
$B_{\widetilde{\mbox{\textbf{SST}}}}$ given in equation (\ref{osst})
below,

\begin{eqnarray}& & E\left[\exp\left(i\omega \sum_{k=1}^{\infty}(\mathbf{Y}_{k}-\mathbf{X}\boldsymbol{\beta }_{k})^{T}
\mathbf{W}_{k}^{T}\boldsymbol{\Lambda}_{k}^{-1}\mathbf{W}_{k}
(\mathbf{Y}_{k}-\mathbf{X}\boldsymbol{\beta }_{k})\right)\right]
\nonumber\\ & &
=\prod_{k=1}^{\infty}\left[\mbox{det}\left(\mathbf{I}_{n\times
n}-2i\omega
\boldsymbol{\Lambda}_{k}^{1/2}\mathbf{W}_{k}^{T}\boldsymbol{\Lambda}_{k}^{-1}\mathbf{W}_{k}\boldsymbol{\Lambda}_{k}^{1/2}\right)\right]^{-1/2}.
\nonumber\\
\label{eqfc1ff}
\end{eqnarray}

\noindent Hence, for $$\|\mathbf{Y}-\mathbf{X}\boldsymbol{\beta }+\mathbf{X}\boldsymbol{\beta }\|^{2}_{\mathbf{W}^{T}\boldsymbol{\Lambda}^{-1}\mathbf{W}}=\sum_{k=1}^{\infty}(\mathbf{Y}_{k}-\mathbf{X}\boldsymbol{\beta
}_{k}+\mathbf{X}\boldsymbol{\beta
}_{k})^{T}
\mathbf{W}_{k}^{T}\boldsymbol{\Lambda}_{k}^{-1}\mathbf{W}_{k}
(\mathbf{Y}_{k}-\mathbf{X}\boldsymbol{\beta }_{k}+\mathbf{X}\boldsymbol{\beta }_{k}),$$ we have
 \begin{eqnarray}& & \hspace*{-0.75cm}F_{\widetilde{\mbox{\textbf{SST}}}}(i\omega )=E\left[\exp\left(i\omega\widetilde{\mbox{\textbf{SST}}}\right)\right]= E\left[\exp\left(i\omega
\|\mathbf{Y}-\mathbf{X}\boldsymbol{\beta }+\mathbf{X}\boldsymbol{\beta }\|^{2}_{\mathbf{W}^{T}\boldsymbol{\Lambda}^{-1}\mathbf{W}}\right)\right]
\nonumber\\
& &\hspace*{-0.75cm}=E\left[\exp\left(i\omega \left\langle
\mathbf{W}^{*}\mathbf{R}_{\boldsymbol{\varepsilon}\boldsymbol{\varepsilon}}^{-1}\mathbf{W}(\mathbf{Y}-\mathbf{X}\boldsymbol{\beta
}),(\mathbf{Y}-\mathbf{X}\boldsymbol{\beta
})\right\rangle_{H^{n}
=\mathcal{H}}\right.\right.\nonumber\\
& &\hspace*{-0.75cm}\left.\left.+2i\omega \left\langle
\mathbf{W}^{*}\mathbf{R}_{\boldsymbol{\varepsilon}\boldsymbol{\varepsilon}}^{-1}\mathbf{W}\mathbf{X}\boldsymbol{\beta
},(\mathbf{Y}-\mathbf{X}\boldsymbol{\beta
})\right\rangle_{H^{n}=\mathcal{H}}\right)\right]\exp\left(
i\omega\sum_{k=1}^{\infty}\boldsymbol{\beta
}_{k}^{T}\mathbf{X}^{T}\mathbf{W}_{k}^{T}\boldsymbol{\Lambda}_{k}^{-1}\mathbf{W}_{k}\mathbf{X}\boldsymbol{\beta
}_{k} \right).\nonumber\\ \label{eqfc1ffv2}
\end{eqnarray}
\noindent  Applying  again Lemma \ref{lem1} with
$\mathbf{Y}-\mathbf{X}\boldsymbol{\beta }$ as $H^{n}$-valued
zero-mean Gaussian vector, and, as before, $\mathbf{M}=2i\omega
\mathbf{W}^{*}\mathbf{R}_{\boldsymbol{\varepsilon}\boldsymbol{\varepsilon}}^{-1}\mathbf{W},$
and with \linebreak $\mathbf{b}=2i\omega
\mathbf{W}^{*}\mathbf{R}_{\boldsymbol{\varepsilon}\boldsymbol{\varepsilon}}^{-1}\mathbf{W}\mathbf{X}\boldsymbol{\beta
},$ keeping in mind that
$\Phi_{k}^{*}\mathbf{W}\Phi_{k}=\mathbf{W}_{k},$
$\Phi_{k}^{*}\mathbf{R}_{\boldsymbol{\varepsilon}\boldsymbol{\varepsilon}}^{-1}\Phi_{k}=\boldsymbol{\Lambda}_{k}^{-1}$
and $\Phi_{k}^{*}(\boldsymbol{\beta })=\boldsymbol{\beta }_{k},$ for
each $k\geq 1,$  we  obtain from (\ref{eqfc1ff}) and
(\ref{eqfc1ffv2}),
\begin{eqnarray}
& & F_{\widetilde{\mbox{\textbf{SST}}}}(i\omega )=
E\left[\exp\left(i\omega\widetilde{\mbox{\textbf{SST}}}\right)\right]=\nonumber\\
& & =\prod_{k=1}^{\infty}\left[\mbox{det}\left(\mathbf{I}_{n\times
n}-2i\omega
\boldsymbol{\Lambda}_{k}^{1/2}\mathbf{W}_{k}^{T}\boldsymbol{\Lambda}_{k}^{-1}\mathbf{W}_{k}\boldsymbol{\Lambda}_{k}^{1/2}\right)\right]^{-1/2}
\nonumber\\
& &\times \exp\left(-4\omega^{2}\sum_{k=1}^{\infty}\boldsymbol{\beta
}_{k}^{T}\mathbf{X}^{T}\mathbf{W}_{k}^{T}\boldsymbol{\Lambda}_{k}^{-1}\mathbf{W}_{k}\boldsymbol{\Lambda}_{k}^{1/2}\right.\nonumber\\
& & \hspace*{3.5cm}\left. \times \left(\mathbf{I}_{n\times
n}-2i\omega
\boldsymbol{\Lambda}_{k}^{1/2}\mathbf{W}_{k}^{T}\boldsymbol{\Lambda}_{k}^{-1}\mathbf{W}_{k}\boldsymbol{\Lambda}_{k}^{1/2}\right)^{-1}
\right.\nonumber\\
& &\hspace*{3.5cm}\left.\times
\boldsymbol{\Lambda}_{k}^{1/2}\mathbf{W}_{k}^{T}\boldsymbol{\Lambda}_{k}^{-1}\mathbf{W}_{k}\mathbf{X}\boldsymbol{\beta
}_{k}\right)\nonumber\\
& &\times \exp\left( i\omega\sum_{k=1}^{\infty}\boldsymbol{\beta
}_{k}^{T}\mathbf{X}^{T}\mathbf{W}_{k}^{T}\boldsymbol{\Lambda}_{k}^{-1}\mathbf{W}_{k}\mathbf{X}\boldsymbol{\beta
}_{k} \right)< \infty,\nonumber\\
\end{eqnarray}
\noindent for $\omega < \min
\{B_{\widetilde{\mbox{\textbf{SST}}}},(1/2)\mathrm{IT}_{\widetilde{\mbox{\textbf{SST}}}}\},$
with $\mathrm{IT}_{\widetilde{\mbox{\textbf{SST}}}}$ being defined
as in (\ref{minsst}), and
\begin{equation}B_{\widetilde{\mbox{\textbf{SST}}}}=\frac{1}{2\max_{k,i}\xi_{i}\left(\boldsymbol{\Lambda}_{k}^{1/2}\mathbf{W}_{k}^{T}\boldsymbol{\Lambda}_{k}^{-1}
\mathbf{W}_{k}\boldsymbol{\Lambda}_{k}^{1/2}
\right)},\label{osst}\end{equation} \noindent where, as before,
$$\xi_{i}\left(\boldsymbol{\Lambda}_{k}^{1/2}\mathbf{W}_{k}^{T}\boldsymbol{\Lambda}_{k}^{-1}
\mathbf{W}_{k}\boldsymbol{\Lambda}_{k}^{1/2} \right)$$ \noindent
denotes the $i$th eigenvalue of matrix
$\boldsymbol{\Lambda}_{k}^{1/2}\mathbf{W}_{k}^{T}\boldsymbol{\Lambda}_{k}^{-1}
\mathbf{W}_{k}\boldsymbol{\Lambda}_{k}^{1/2},$ for $i=1,\dots,n,$
and for each $k\geq 1.$ An analytic continuation argument (see \cite{Lukacs}, Th. 7.1.1)
guarantees that $F_{\widetilde{\mbox{\textbf{SST}}}}(i\omega )$
defines the unique limit characteristic functional  for all  values
of $\omega.$
\item[] (ii) Similarly, from Lemma \ref{lem1},  for suitable   $\omega $ (see equation (\ref{scomega}) below), we have
\begin{eqnarray}& & E\left[\exp\left(i\omega \sum_{k=1}^{\infty}(\mathbf{Y}_{k}-\mathbf{X}\boldsymbol{\beta }_{k})^{T}
\mathbf{W}_{k}^{T}\boldsymbol{\Lambda
}_{k}^{-1}\mathbf{X}\right.\right.\nonumber\\
& & \hspace*{3cm}\left.\left. \times
(\mathbf{X}^{T}\boldsymbol{\Lambda
}_{k}^{-1}\mathbf{X})^{-1}\mathbf{X}^{T}\boldsymbol{\Lambda
}_{k}^{-1}\mathbf{W}_{k}(\mathbf{Y}_{k}-\mathbf{X}\boldsymbol{\beta
}_{k})\right)\right] \nonumber\\ & &
=\prod_{k=1}^{\infty}\left[\mbox{det}\left(\mathbf{I}_{n\times
n}-2i\omega
\boldsymbol{\Lambda}_{k}^{1/2}\mathbf{W}_{k}^{T}\boldsymbol{\Lambda
}_{k}^{-1}\mathbf{X}\right.\right.\nonumber\\ & &
\hspace*{3cm}\left.\left.\times (\mathbf{X}^{T}\boldsymbol{\Lambda
}_{k}^{-1}\mathbf{X})^{-1}\mathbf{X}^{T}\boldsymbol{\Lambda
}_{k}^{-1}\mathbf{W}_{k}\boldsymbol{\Lambda}_{k}^{1/2}\right)\right]^{-1/2}.
\nonumber\\
\label{eqfc1ffii}
\end{eqnarray}
Applying again  Lemma \ref{lem1} with
$$\mathbf{M}=2i\omega \mathbf{W}^{*}\mathbf{R}_{\boldsymbol{\varepsilon}\boldsymbol{\varepsilon}}^{-1}\mathbf{X}(\mathbf{X}^{T}\mathbf{R}_{\boldsymbol{\varepsilon}\boldsymbol{\varepsilon}}^{-1}\mathbf{X})^{-1}
\mathbf{X}^{T}\mathbf{R}_{\boldsymbol{\varepsilon}\boldsymbol{\varepsilon}}^{-1}\mathbf{W},$$\noindent
and with
$$\mathbf{b}=2i\omega \mathbf{W}^{*}\mathbf{R}_{\boldsymbol{\varepsilon}\boldsymbol{\varepsilon}}^{-1}\mathbf{X}(\mathbf{X}^{T}\mathbf{R}_{\boldsymbol{\varepsilon}\boldsymbol{\varepsilon}}^{-1}\mathbf{X})^{-1}
\mathbf{X}^{T}\mathbf{R}_{\boldsymbol{\varepsilon}\boldsymbol{\varepsilon}}^{-1}\mathbf{W}\mathbf{X}\boldsymbol{\beta
},$$ \noindent one can get,
\begin{eqnarray}& &
F_{\widetilde{\mbox{\textbf{SSR}}}}(i\omega
)=E\left[\exp\left(i\omega
\widetilde{\mbox{\textbf{SSR}}}\right)\right]\nonumber\\
& &  =\prod_{k=1}^{\infty}\left[\mbox{det}\left(\mathbf{I}_{n\times
n}-2i\omega
\boldsymbol{\Lambda}_{k}^{1/2}\mathbf{W}_{k}^{T}\boldsymbol{\Lambda
}_{k}^{-1}\mathbf{X}(\mathbf{X}^{T}\boldsymbol{\Lambda
}_{k}^{-1}\mathbf{X})^{-1}\mathbf{X}^{T}\boldsymbol{\Lambda
}_{k}^{-1}\mathbf{W}_{k}\boldsymbol{\Lambda}_{k}^{1/2}\right)\right]^{-1/2}
\nonumber\\
& &\times \exp\left(-4\omega^{2}\sum_{k=1}^{\infty}\boldsymbol{\beta
}_{k}^{T}\mathbf{X}^{T}\mathbf{W}_{k}^{T}\boldsymbol{\Lambda
}_{k}^{-1}\mathbf{X}(\mathbf{X}^{T}\boldsymbol{\Lambda
}_{k}^{-1}\mathbf{X})^{-1}\mathbf{X}^{T}\boldsymbol{\Lambda
}_{k}^{-1}\mathbf{W}_{k}\boldsymbol{\Lambda}_{k}^{1/2}\right.\nonumber\\
& & \hspace*{1.5cm}\left.\times \left(\mathbf{I}_{n\times
n}-2i\omega
\boldsymbol{\Lambda}_{k}^{1/2}\mathbf{W}_{k}^{T}\boldsymbol{\Lambda
}_{k}^{-1}\mathbf{X}(\mathbf{X}^{T}\boldsymbol{\Lambda
}_{k}^{-1}\mathbf{X})^{-1}\mathbf{X}^{T}\boldsymbol{\Lambda
}_{k}^{-1}\mathbf{W}_{k}\boldsymbol{\Lambda}_{k}^{1/2}\right)^{-1}
\right.\nonumber\\
& &\hspace*{1.5cm}\left.\times
\boldsymbol{\Lambda}_{k}^{1/2}\mathbf{W}_{k}^{T}\boldsymbol{\Lambda
}_{k}^{-1}\mathbf{X}(\mathbf{X}^{T}\boldsymbol{\Lambda
}_{k}^{-1}\mathbf{X})^{-1}\mathbf{X}^{T}\boldsymbol{\Lambda
}_{k}^{-1}\mathbf{W}_{k}\mathbf{X}\boldsymbol{\beta
}_{k}\right)\nonumber\\
& &\hspace*{1.5cm}\times  \exp\left(
i\omega\sum_{k=1}^{\infty}\boldsymbol{\beta
}_{k}^{T}\mathbf{X}^{T}\mathbf{W}_{k}^{T}\boldsymbol{\Lambda
}_{k}^{-1}\mathbf{X}(\mathbf{X}^{T}\boldsymbol{\Lambda
}_{k}^{-1}\mathbf{X})^{-1}\mathbf{X}^{T}\boldsymbol{\Lambda
}_{k}^{-1}\mathbf{W}_{k}\mathbf{X}\boldsymbol{\beta
}_{k} \right)<\infty,\nonumber\\
\label{eqfc1ffiibb}
\end{eqnarray}
for $\omega < \min
\{B_{\widetilde{\mbox{\textbf{SSR}}}},\mathrm{IT}_{\widetilde{\mbox{\textbf{SSR}}}}\},$
with
\begin{eqnarray}
& &
\mathrm{IT}_{\widetilde{\mbox{\textbf{SSR}}}}=\frac{1}{2\mbox{trace}\left(\mathbf{W}^{*}\boldsymbol{R}_{\boldsymbol{\varepsilon}\boldsymbol{\varepsilon}}^{-1}\mathbf{X}(\mathbf{X}^{T}\boldsymbol{R}_{\boldsymbol{\varepsilon}\boldsymbol{\varepsilon}}^{-1}
\mathbf{X})^{-1}\mathbf{X}^{T}\boldsymbol{R}_{\boldsymbol{\varepsilon}\boldsymbol{\varepsilon}}^{-1}
\mathbf{W}\boldsymbol{R}_{\boldsymbol{\varepsilon}\boldsymbol{\varepsilon}}\right)}
\nonumber\\ &
&B_{\widetilde{\mbox{\textbf{SSR}}}}=\frac{1}{2\max_{k,i}\xi_{i}\left(\boldsymbol{\Lambda}_{k}^{1/2}\mathbf{W}_{k}^{T}\boldsymbol{\Lambda
}_{k}^{-1}\mathbf{X}(\mathbf{X}^{T}\boldsymbol{\Lambda
}_{k}^{-1}\mathbf{X})^{-1}\mathbf{X}^{T}\boldsymbol{\Lambda
}_{k}^{-1}\mathbf{W}_{k}\boldsymbol{\Lambda}_{k}^{1/2}\right)},\nonumber\\
\label{scomega}
\end{eqnarray}
\noindent where, as before, for $i=1,\dots,n,$
$$\xi_{i}\left(\boldsymbol{\Lambda}_{k}^{1/2}\mathbf{W}_{k}^{T}\boldsymbol{\Lambda
}_{k}^{-1}\mathbf{X}(\mathbf{X}^{T}\boldsymbol{\Lambda
}_{k}^{-1}\mathbf{X})^{-1}\mathbf{X}^{T}\boldsymbol{\Lambda
}_{k}^{-1}\mathbf{W}_{k}\boldsymbol{\Lambda}_{k}^{1/2}\right)$$
\noindent denotes the $i$th eigenvalue of matrix
$$\boldsymbol{\Lambda}_{k}^{1/2}\mathbf{W}_{k}^{T}\boldsymbol{\Lambda
}_{k}^{-1}\mathbf{X}(\mathbf{X}^{T}\boldsymbol{\Lambda
}_{k}^{-1}\mathbf{X})^{-1}\mathbf{X}^{T}\boldsymbol{\Lambda
}_{k}^{-1}\mathbf{W}_{k}\boldsymbol{\Lambda}_{k}^{1/2},$$ \noindent
for each $k\geq 1.$ An analytic continuation argument (see \cite{Lukacs}, Th. 7.1.1)
guarantees that $F_{\widetilde{\mbox{\textbf{SSR}}}}(i\omega )$
defines the unique limit characteristic functional  for all  values
of $\omega.$

\item[] (iii) From Lemma \ref{lem1}, for suitable  $\omega $  (see equation (\ref{eqowc3})
below), we obtain

\begin{eqnarray}& & E\left[\exp\left(i\omega \sum_{k=1}^{\infty}(\mathbf{Y}_{k}-\mathbf{X}\boldsymbol{\beta }_{k})^{T}\left(\mathbf{W}_{k}^{T}\boldsymbol{\Lambda
}_{k}^{-1}\mathbf{W}_{k}-\mathbf{W}_{k}^{T}\boldsymbol{\Lambda
}_{k}^{-1}\mathbf{X}(\mathbf{X}^{T}\boldsymbol{\Lambda
}_{k}^{-1}\mathbf{X})^{-1}\right.\right.\right.
\nonumber\\
& &  \hspace*{3cm}\left.\left.\left.\times
\mathbf{X}^{T}\boldsymbol{\Lambda
}_{k}^{-1}\mathbf{W}_{k}\right)(\mathbf{Y}_{k}-\mathbf{X}\boldsymbol{\beta
}_{k})\right)\right]
\nonumber\\
& & = \prod_{k=1}^{\infty}\left[\mbox{det}\left(\mathbf{I}_{n\times
n}-2i\omega
\boldsymbol{\Lambda}_{k}^{1/2}\left(\mathbf{W}_{k}^{T}\boldsymbol{\Lambda
}_{k}^{-1}\mathbf{W}_{k}-\mathbf{W}_{k}^{T}\boldsymbol{\Lambda
}_{k}^{-1}\mathbf{X}(\mathbf{X}^{T}\boldsymbol{\Lambda
}_{k}^{-1}\mathbf{X})^{-1}\right.\right.\right.
\nonumber\\
& &  \hspace*{3cm}\left.\left.\left.\times
\mathbf{X}^{T}\boldsymbol{\Lambda
}_{k}^{-1}\mathbf{W}_{k}\right)\boldsymbol{\Lambda}_{k}^{1/2}\right)\right]^{-1/2}.
\label{eqfc1ffiicc}
\end{eqnarray}

As before, considering again  Lemma \ref{lem1} with
$$\mathbf{M}=2i\omega \left(\mathbf{W}^{*}\mathbf{R}_{\boldsymbol{\varepsilon}\boldsymbol{\varepsilon}}^{-1}\mathbf{W}-\mathbf{W}^{*}\mathbf{R}_{\boldsymbol{\varepsilon}\boldsymbol{\varepsilon}}^{-1}
\mathbf{X}(\mathbf{X}^{T}\mathbf{R}_{\boldsymbol{\varepsilon}\boldsymbol{\varepsilon}}^{-1}\mathbf{X})^{-1}
\mathbf{X}^{T}\mathbf{R}_{\boldsymbol{\varepsilon}\boldsymbol{\varepsilon}}^{-1}\mathbf{W}\right),$$\noindent
and with
$$\mathbf{b}=2i\omega \left(\mathbf{W}^{*}\mathbf{R}_{\boldsymbol{\varepsilon}\boldsymbol{\varepsilon}}^{-1}\mathbf{W}-\mathbf{W}^{*}\mathbf{R}_{\boldsymbol{\varepsilon}\boldsymbol{\varepsilon}}^{-1}
\mathbf{X}(\mathbf{X}^{T}\mathbf{R}_{\boldsymbol{\varepsilon}\boldsymbol{\varepsilon}}^{-1}\mathbf{X})^{-1}
\mathbf{X}^{T}\mathbf{R}_{\boldsymbol{\varepsilon}\boldsymbol{\varepsilon}}^{-1}\mathbf{W}\right)\mathbf{X}\boldsymbol{\beta
},$$ \noindent we obtain,
\begin{eqnarray}
& & F_{\widetilde{\mbox{\textbf{SSE}}}}(i\omega
)=E\left[\exp\left(i\omega
\widetilde{\mbox{\textbf{SSE}}}\right)\right]\nonumber\\
& &  =\prod_{k=1}^{\infty}\left[\mbox{det}\left(\mathbf{I}_{n\times
n}-2i\omega
\boldsymbol{\Lambda}_{k}^{1/2}\left(\mathbf{W}^{T}_{k}\boldsymbol{\Lambda}_{k}^{-1}\mathbf{W}_{k}\right.\right.\right.
\nonumber\\
& &\hspace*{3cm}\left.\left.\left.
-\mathbf{W}^{T}_{k}\boldsymbol{\Lambda}_{k}^{-1}
\mathbf{X}(\mathbf{X}^{T}\boldsymbol{\Lambda}_{k}^{-1}\mathbf{X})^{-1}
\mathbf{X}^{T}\boldsymbol{\Lambda}_{k}^{-1}\mathbf{W}_{k}\right)\boldsymbol{\Lambda}_{k}^{1/2}\right)\right]^{-1/2}
\nonumber
\end{eqnarray}
\begin{eqnarray}
& &\times \exp\left(-4\omega^{2}\sum_{k=1}^{\infty}\boldsymbol{\beta
}_{k}^{T}\mathbf{X}^{T}\left(\mathbf{W}^{T}_{k}\boldsymbol{\Lambda}_{k}^{-1}\mathbf{W}_{k}
\right.\right.\nonumber\\ & & \hspace*{3cm} \left.\left.
-\mathbf{W}^{T}_{k}\boldsymbol{\Lambda}_{k}^{-1}
\mathbf{X}(\mathbf{X}^{T}\boldsymbol{\Lambda}_{k}^{-1}\mathbf{X})^{-1}
\mathbf{X}^{T}\boldsymbol{\Lambda}_{k}^{-1}\mathbf{W}_{k}\right)\boldsymbol{\Lambda}_{k}^{1/2}\right.\nonumber\\
& & \left.\times \left(\mathbf{I}_{n\times n}-2i\omega
\boldsymbol{\Lambda}_{k}^{1/2}\left(\mathbf{W}^{T}_{k}\boldsymbol{\Lambda}_{k}^{-1}\mathbf{W}_{k}\right.\right.\right.
\nonumber\\
& & \hspace*{3cm}\left.\left.\left.
-\mathbf{W}^{T}_{k}\boldsymbol{\Lambda}_{k}^{-1}
\mathbf{X}(\mathbf{X}^{T}\boldsymbol{\Lambda}_{k}^{-1}\mathbf{X})^{-1}
\mathbf{X}^{T}\boldsymbol{\Lambda}_{k}^{-1}\mathbf{W}_{k}\right)\boldsymbol{\Lambda}_{k}^{1/2}\right)^{-1}
\right.\nonumber\\
& &\hspace*{0.5cm}\left.\times
\boldsymbol{\Lambda}_{k}^{1/2}\left(\mathbf{W}^{T}_{k}\boldsymbol{\Lambda}_{k}^{-1}\mathbf{W}_{k}-\mathbf{W}^{T}_{k}\boldsymbol{\Lambda}_{k}^{-1}
\mathbf{X}(\mathbf{X}^{T}\boldsymbol{\Lambda}_{k}^{-1}\mathbf{X})^{-1}
\mathbf{X}^{T}\boldsymbol{\Lambda}_{k}^{-1}\mathbf{W}_{k}\right)\mathbf{X}\boldsymbol{\beta
}_{k}\right)\nonumber\\
& &\times  \exp\left( i\omega\sum_{k=1}^{\infty}\boldsymbol{\beta
}_{k}^{T}\mathbf{X}^{T}\left(\mathbf{W}^{T}_{k}\boldsymbol{\Lambda}_{k}^{-1}\mathbf{W}_{k}-\mathbf{W}^{T}_{k}\boldsymbol{\Lambda}_{k}^{-1}
\mathbf{X}(\mathbf{X}^{T}\boldsymbol{\Lambda}_{k}^{-1}\mathbf{X})^{-1}\right.\right.\nonumber\\
& &\hspace*{4cm}\left.\left.\times
\mathbf{X}^{T}\boldsymbol{\Lambda}_{k}^{-1}\mathbf{W}_{k}\right)\mathbf{X}\boldsymbol{\beta
}_{k} \right)<\infty,\nonumber\\
\label{eqfc1ffiidd}
\end{eqnarray}
\noindent for $\omega < \min
\{B_{\widetilde{\mbox{\textbf{SSE}}}},\mathrm{IT}_{\widetilde{\mbox{\textbf{SSE}}}}\},$
with \begin{eqnarray} & &
\mathrm{IT}_{\widetilde{\mbox{\textbf{SSE}}}}=
\frac{1}{2\mbox{trace}\left(\left(\mathbf{W}^{*}\boldsymbol{R}_{\boldsymbol{\varepsilon}\boldsymbol{\varepsilon}}^{-1}\mathbf{W}-\mathbf{W}^{*}
\boldsymbol{R}_{\boldsymbol{\varepsilon}\boldsymbol{\varepsilon}}^{-1}\mathbf{X}(\mathbf{X}^{T}\boldsymbol{R}_{\boldsymbol{\varepsilon}\boldsymbol{\varepsilon}}^{-1}\mathbf{X})^{-1}\mathbf{X}^{T}
\boldsymbol{R}_{\boldsymbol{\varepsilon}\boldsymbol{\varepsilon}}^{-1}\mathbf{W}\right)
\boldsymbol{R}_{\boldsymbol{\varepsilon}\boldsymbol{\varepsilon}}\right)}\nonumber\\
& &
B_{\widetilde{\mbox{\textbf{SSE}}}}=\nonumber\\
& &
\frac{1}{2\max_{k,i}\xi_{i}\left(\boldsymbol{\Lambda}_{k}^{1/2}\left(\mathbf{W}^{T}_{k}
\boldsymbol{\Lambda}_{k}^{-1}\mathbf{W}_{k}-\mathbf{W}^{T}_{k}\boldsymbol{\Lambda}_{k}^{-1}
\mathbf{X}(\mathbf{X}^{T}\boldsymbol{\Lambda}_{k}^{-1}\mathbf{X})^{-1}
\mathbf{X}^{T}\boldsymbol{\Lambda}_{k}^{-1}\mathbf{W}_{k}\right)\boldsymbol{\Lambda}_{k}^{1/2}\right)},\nonumber\\
\label{eqowc3} \end{eqnarray}\noindent where, as before, for
$i=1,\dots,n,$
$$\xi_{i}\left(\boldsymbol{\Lambda}_{k}^{1/2}\left(\mathbf{W}^{T}_{k}
\boldsymbol{\Lambda}_{k}^{-1}\mathbf{W}_{k}-\mathbf{W}^{T}_{k}\boldsymbol{\Lambda}_{k}^{-1}
\mathbf{X}(\mathbf{X}^{T}\boldsymbol{\Lambda}_{k}^{-1}\mathbf{X})^{-1}
\mathbf{X}^{T}\boldsymbol{\Lambda}_{k}^{-1}\mathbf{W}_{k}\right)\boldsymbol{\Lambda}_{k}^{1/2}\right)$$
\noindent denotes the $i$th eigenvalue of matrix
$$\boldsymbol{\Lambda}_{k}^{1/2}\left(\mathbf{W}^{T}_{k}
\boldsymbol{\Lambda}_{k}^{-1}\mathbf{W}_{k}-\mathbf{W}^{T}_{k}\boldsymbol{\Lambda}_{k}^{-1}
\mathbf{X}(\mathbf{X}^{T}\boldsymbol{\Lambda}_{k}^{-1}\mathbf{X})^{-1}
\mathbf{X}^{T}\boldsymbol{\Lambda}_{k}^{-1}\mathbf{W}_{k}\right)\boldsymbol{\Lambda}_{k}^{1/2},$$
\noindent for each $k\geq 1.$ An analytic continuation argument (see \cite{Lukacs}, Th. 7.1.1)
guarantees that $F_{\widetilde{\mbox{\textbf{SSE}}}}(i\omega )$
defines the unique limit characteristic functional  for all values
of $\omega.$
\end{itemize}

\end{document}